\documentclass[11pt]{amsart}
\usepackage{amsmath,amsfonts,amssymb,amsthm,graphicx,tikz,mathtools,dsfont,hyperref,cite, comment}
\usetikzlibrary{arrows,chains,matrix,positioning,scopes}

\newtheorem{thm}{Theorem}
\newtheorem{dfn}[thm]{Definition}
\newtheorem{lem}[thm]{Lemma}
\newtheorem{prop}[thm]{Proposition}
\newtheorem{remark}[thm]{Remark}

\newtheorem{ex}[thm]{Example}

\newtheorem*{thm*}{Theorem}
\newtheorem*{dfn*}{Definition}
\newtheorem*{lem*}{Lemma}
\newtheorem*{prop*}{Proposition}
\newtheorem*{remark*}{Remark}
\newtheorem*{cor*}{Corollary}
\newtheorem*{ex*}{Example}
\newtheorem*{question*}{Questions}
\newtheorem*{exercise*}{Exercise}

\def\Z{\mathbb{Z}}

\def\N{\mathbb{N}}

\def\L{\mathcal{L}}

\numberwithin{equation}{section}
\numberwithin{thm}{section}

\makeatletter
\tikzset{join/.code=\tikzset{after node path={%
			\ifx\tikzchainprevious\pgfutil@empty\else(\tikzchainprevious)%
			edge[every join]#1(\tikzchaincurrent)\fi}}}
\makeatother

\tikzset{>=stealth',every on chain/.append style={join},
	every join/.style={->}}
\tikzstyle{labeled}=[execute at begin node=$\scriptstyle,
execute at end node=$]
\title[Topological recursion for twisted Higgs bundles]{Topological recursion and variations of spectral curves for twisted Higgs bundles}
\date{\today}

\author{Christopher Mahadeo}
\address{Department of Mathematics, Statistics, and Computer Science, University of Illinois at Chicago, IL, USA~ 60607}
\email{cmahadeo@uic.edu}

\author{Steven Rayan}
\address{Centre for Quantum Topology and Its Applications (quanTA) and Department of Mathematics and Statistics, University of Saskatchewan, SK, Canada~ S7N 5E6}
\email{rayan@math.usask.ca}

\begin{document}

\maketitle

\begin{abstract}

Prior works relating meromorphic Higgs bundles to topological recursion, in particular those of Dumitrescu-Mulase, have considered non-singular models that allow the recursion to be carried out on a smooth Riemann surface.  We start from an $\mathcal{L}$-twisted Higgs bundle for some fixed holomorphic line bundle $\mathcal{L}$ on the surface.  We decorate the Higgs bundle with the choice of a section $s$ of $K^*\otimes\mathcal{L}$, where $K$ is the canonical line bundle, and then encode this data as a $b$-structure on the base Riemann surface which lifts to the associated Hitchin spectral curve.  We then propose a so-called twisted topological recursion on the spectral curve, after which the corresponding Eynard-Orantin differentials live in a twisted cotangent bundle.  This formulation retains, and interacts explicitly with, the singular structure of the original meromorphic setting --- equivalently, the zero divisor of $s$ --- while performing the recursion.  Finally, we show that the $g=0$ twisted Eynard-Orantin differentials compute the Taylor expansion of the period matrix of the spectral curve, mirroring a result of Baraglia-Huang for ordinary Higgs bundles and topological recursion.  Starting from the spectral curve as a polynomial form in an affine coordinate rather than a Higgs bundle, our result implies that, under certain conditions on $s$, the expansion is independent of the ambient space $\mbox{Tot}(\mathcal{L})$ in which the curve is interpreted to reside.
\end{abstract}

\tableofcontents

\let\thefootnote\relax\footnotetext{2020\textit{ Mathematics Subject Classification.} 14D20, 70H06, 14A21.}\let\thefootnote\relax

\section{Introduction}

Topological recursion is a mathematical procedure that can be interpreted in many ways, one of which is as a mechanism for turning a spectral curve from a classical integrable system into its quantum counterpart.  This procedure has been demonstrated to simultaneously solve problems in enumerative geometry and produce exact solutions to the Schr\"odinger equation defined by the quantum curve.  As such, topological recursion has received much attention in both the theoretical physics and algebro-symplectic geometry communities.  At the same time, Higgs bundles are geometric objects arising from physics by way of the dimensionally-reduced self-dual Yang-Mills equations, now known as the Hitchin equations.  The moduli space of these objects is a completely integrable Hamiltonian system with spectral curves arising directly from the characteristic data of the Higgs bundles themselves.  One can then ask about Hitchin spectral curves as input data for topological recursion.  This has been explored in a number of papers by Dumitrescu-Mulase \cite{DumitrescuMulase14, DumitrescuMulase14b, DumitrescuMulase17, DumitrescuMulase18} as well as Baraglia-Huang \cite{BaragliaHuang17}.

\subsection{Higgs bundles}

A \emph{Higgs bundle} is a pair $(\mathcal{E}, \phi)$, where $\mathcal{E}$ is a holomorphic vector bundle with a one-form valued endomorphism $\phi\in H^{0}(\text{End}\mathcal{E}\otimes K)$, called a \emph{Higgs field}.  Appearing initially as solutions to a dimensionally-reduced version of the self-dual Yang-Mills equations, they can be generalized by allowing $\phi$ to have poles (\emph{cf.} \cite{BodenYokogawa96, DumitrescuMulase14b, RayanSchaposnik2020}) or by allowing $\phi$ to take values in another holomorphic line bundle $\mathcal{L}$ (\emph{cf.}\cite{Nitsure91, Rayan13}).  Under sufficient stability conditions, we can consider the moduli space of Higgs bundles
\begin{equation*}
    \mathcal{M}_{X}^{\mathcal{L}} = \frac{\Big\{ \text{stable}\ \mathcal{L}\text{-twisted Higgs bundles} \Big\}}{\text{conjugation}}.
\end{equation*}
The moduli space of Higgs bundles $\mathcal{M}_{X}^{\mathcal{L}}$ contains rich geometry, as it is a completely integrable Hamiltonian system, and admits a hyperk\"ahler structure in the ordinary ($\mathcal{L} = K$) setting.    

\subsection{Topological recursion}

Topological recursion, originally introduced in \cite{ChekhovEynardOrantin06, EynardOrantin07b, EynardOrantin08}, is a recursive formula that associates to a spectral curve $S$, which is a complex algebraic curve arising as the spectrum of a matrix-valued function with additional conditions, a family of multi-differentials $W_{g,n}$.  These \emph{Eynard-Orantin differentials} are built out of canonical geometric data of the spectral curve and its embedding into the cotangent bundle of a base curve $K_{X}$ by
\begin{itemize}
    \item $W_{0,1}$ is a meromorphic $1$-form on $S$ (typically chosen to be the tautological section of $K_{X}$); and
    \item $W_{0,2}(z_1,z_2) = B(z_{1},z_{2})$,
\end{itemize}
with the remaining terms of $2g - 2 + n \geq 0$, being defined recursively by
\begin{align*}
    W_{g,n+1}(z_0, \textbf{z}) = \sum_{p\in R}&\mathrm{Res}_{z = p}K_{p}(z_0,z)\Big[ W_{g-1,n+2}(z, \sigma_{p}(z),\textbf{z}) \\
    \nonumber &+\sum_{\substack{g_1+g_2=g \\ I\cup J = \textbf{z}}}^{'}W_{g_1, |I|+1}(z,I)W_{g_2, |J|+1}(\sigma_{p}(z), J)\Big],
\end{align*}
where the prime signifies summation excluding the cases $(g_1,I)$ or $(g_2,J) = (0,0)$.\\

Topological recursion was developed in the context of matrix models and random matrix theory, but has since shown close ties to problems and fundamental structures in enumerative geometry.  From this recursion procedure, it is possible to recover many known invariants: Weil-Peterson volumes (Mirzakhani's recursion) \cite{Mirzakhani07, EynardOrantin07}, Hurwitz numbers \cite{GouldenJacksonVainshtein00, EynardMulaseSafnuk11, BouchardMarino08}, Virasoro constraints for two-dimensional gravity \cite{DijkgraafVerlinde91}, and Tutte's enumeration of maps \cite{Eynard04, AlexandrovMironovMorozov05}.  There are conjectures that topological recursion is also related to knot invariants \cite{DijkgraafFujiManabe11, BorotEynard12}.  As such, topological recursion sheds light on a myriad of deep and mysterious connections between topology, algebraic and differential geometry, representation theory, combinatorics, and physics.\\

From the Higgs field, we can produce a \emph{Hitchin spectral curve} by looking at the zero variety of its characteristic equation (when interpreted correctly).  When this spectral curve satisfies the correct properties, it becomes a candidate curve on which to apply topological recursion.  Recent work has been done to understand the relationship between topological recursion and Higgs bundles.  Dumitrescu-Mulase \cite{DumitrescuMulase14, DumitrescuMulase14b, DumitrescuMulase17, DumitrescuMulase18} have looked at generalizing topological recursion to a larger class of Hitchin spectral curves and the quantization of these spectral curves, while Baraglia-Huang \cite{BaragliaHuang17}, Bertola-Korotkin \cite{BertolaKorotkin18} and Chaimanwong et al. \cite{ChaimanowongNorbury20} have investigated the relationship between topological recursion on Hitchin spectral curves and geometric properties of the moduli space of Higgs bundles.

\subsection{Geometry of Hitchin moduli spaces}

The complex integrable system structure of the ordinary Hitchin moduli space $\mathcal{M}_{X}$ gives rise to a special K\"ahler structure on the Hitchin base.  This special K\"ahler structure can be written in terms of the period matrix of the spectral curve at the point in the Hitchin base.  The special K\"ahler metric combines with a metric along the fibres of the moduli space to produce the semi-flat metric on the regular locus of $\mathcal{M}_{X}$.  This can be thought of as an approximation of the complete hyperk\"ahler metric.  Baraglia-Huang \cite{BaragliaHuang17} show that the Taylor series expansion of the period matrix (and by extension, information about the hyperk\"ahler metric) about a point in the base can be computed using the $g=0$ Eynard-Orantin differentials on the Hitchin spectral curve associated to that point.

\begin{thm}[Baraglia-Huang,\cite{BaragliaHuang17}]
    \begin{equation}        \partial_{i_{1}}\partial_{i_{2}}\dots\partial_{i_{m-2}}\tau_{i_{m-1}i_{m}} = -\left( \frac{i}{2\pi} \right)^{m-1} \int_{p_{i_{1}}\in b_{i_{1}}}\dots\int_{p_{m}\in b_{i_{m}}} W_{0,m}(p_{1},...,p_{m})
    \end{equation}
\end{thm}

\subsection{Geometry of $\mathcal{L}$-twisted Hitchin spectral curves}

Notable changes occur when considering $\mathcal{L}$-twisted Higgs bundles in lieu of ordinary Higgs bundles.  In the twisted moduli space, the Hitchin base and the fibres of the moduli space no longer have the same dimension.  This means we cannot view the Hitchin base as the space of deformation of the spectral curve as was done in the $\mathcal{L}=K$ setting.  In the ordinary setting, the tautological section $\eta$ (which is used to define Hitchin spectral curves) is related to the canonical symplectic structure on the cotangent bundle, however in the twisted setting, there is no longer a canonical symplectic structure, and so $\eta$ is viewed purely as an algebraic object valued in $\pi^{*}\mathcal{L}$.  The lack of a canonical symplectic structure also removes a ``nice'' canonical coordinate system in which to work.\\

We study the hypercohomology of stable $\mathcal{L}$-twisted Higgs bundles $(\mathcal{E}, \phi)$ on a Riemann surface $X$ with spectral curve $S$ given by the two double complexes
\begin{align*}
    D &= (\delta, \wedge\phi),\\
    D' &= (\wedge\phi, \delta),
\end{align*}
where $\wedge\phi$ is the differential coming from the Higgs field, and $\delta$ is the \v Cech differential.  We find a suitable expression for the tangent space of the fibres,
\begin{align*}
    T_{(\mathcal{E},\phi)}\mathcal{M}_{X}^{\mathcal{L}}(r,d) &\cong T_{\mathcal{Q}}Jac(S)\times T_{h(\mathcal{E},\phi)}\mathcal{B}\\
    &= H^{1}\left(\bigoplus^{r-1}_{i=0} \mathcal{L}^{-i}\right)\times \mathcal{B}.
\end{align*}
We then define a so-called effective base, which is dual to $T_{Q}Jac(S)$.
\begin{dfn}
We call $\mathcal{B}_{eff} \coloneqq H^{0}\left(\bigoplus^{r-1}_{i=0} \mathcal{L}^{i}\otimes K\right)$ the \textbf{effective Hitchin base}.
\end{dfn}

Choosing a section $s\in H^{0}(X,K^{*}\otimes\mathcal{L})$, we reframe our $\mathcal{L}$-twisted Higgs bundle as a $K(Z)$-valued Higgs bundle, where $Z$ is the zero-divisor of $s$.  This imposes the structure of $b$-geometry onto $X$ and $S$, and thus a log-symplectic structure.  In this $b$-geometric picture, we define the twisted Eynard-Orantin invariants (for suitably defined $\widehat{B}$ and modified $K_{p}$ below).

\begin{dfn}
The \textbf{$\mathcal{L}$-twisted Eynard-Orantin differentials} $W_{g,n}$ are meromorphic sections of the $n$-th exterior tensor product $K_{S}(Z)^{\boxtimes n}$, i.e. multi-$b$-differentials, defined as follows:\\

The initial conditions of the recursion are given by:
\begin{align}
    &W_{0,1}(z) = y(z)\frac{dx(z)}{x(z)}\\
    &W_{0,2}(z_1,z_2) = \widehat{B}(z_{1},z_{2}).
\end{align}
For all $g,n\in \N$ and $2g - 2 + n \geq 0$, define $W_{g,n}$ recursively by
\begin{align}
    W_{g,n+1}(z_0, \textbf{z}) = \sum_{p\in R}&\mathrm{Res}_{z = p}K_{p}(z_0,z)\Big[ W_{g-1,n+2}(z, \sigma_{p}(z),\textbf{z})\\
    \nonumber &+\sum_{\substack{g_1+g_2=g \\ I\cup J = \textbf{z}}}^{'}W_{g_1, |I|+1}(z,I)W_{g_2, |J|+1}(\sigma_{p}(z), J)\Big] 
\end{align}
where the prime signifies summation excluding the cases $(g_1,I)$ or $(g_2,J) = (0,0)$.\\
\end{dfn}

We argue that local computations of the twisted Eynard-Orantin differentials mirror the $\mathcal{L}=K$ setting, so the twisted differentials satisfy the same properties as the ordinary ones.  In particular, they also satisfy a variational formula.

\begin{thm}[Variational Formula for twisted-E-O invariants]
For $g + k > 1$,
\begin{equation}
    \delta_{i}W_{g,k}(p_{1},...,p_{k}) = -\frac{1}{2\pi i} \int_{p\in b_{i}} W_{g,k+1}(p,p_{1},...,p_{k}),
\end{equation}
where the cycle $b_{i}$ is chosen so that it contains no ramification points.
\end{thm}

Using this variational formula, we prove that the $g=0$ twisted Eynard-Orantin differentials compute the Taylor expansion of the period matrix of $S$ in the image of the effective Hitchin base.

\begin{thm}\label{ThmMainIntro}
    \begin{equation}        \partial_{i_{1}}\partial_{i_{2}}\dots\partial_{i_{m-2}}\tau_{i_{m-1}i_{m}} = -\left( \frac{i}{2\pi} \right)^{m-1} \int_{p_{i_{1}}\in b_{i_{1}}}\dots\int_{p_{m}\in b_{i_{m}}} W_{0,m}(p_{1},...,p_{m}).
    \end{equation}
\end{thm}

\subsection{Overview}

We will begin by establishing some foundational theory required to engage in the later sections.  In \emph{Section 2}, we review basic properties of Higgs bundles.  We approach this topic in both the usual case and the $\mathcal{L}$-twisted case. We then in \emph{Section 3} introduce the Eynard-Orantin topological recursion, and topological recursion for Hitchin spectral curves in the sense of Dumitrescu-Mulase.  Both sections are intended to be somewhat conversational primers as we do not assume intimate knowledge of Higgs bundles and the structure of Hitchin moduli spaces amongst readers coming from the topological recursion community or vice-versa. In \emph{Section 4}, we study the deformation associated with the $\mathcal{L}$-twisted moduli space, and define a ``twisted'' version of topological recursion.  We then explore the relationship between topological recursion and the the Taylor expansion of the period matrix of a given local spectral curve, culminating in Theorem \eqref{ThmMainIntro}.  Finally, in \emph{Section 5}, we discuss the further aims of the work, which includes further speculations about the geometry of the $\mathcal{L}$-twisted setting.\\

\noindent\emph{Acknowledgements.} The authors thank David Baraglia, Rapha\"el Belliard, Vincent Bouchard, Gaetan Borot, Nitin Kumar Chidambaram, Norman Do, Ron Donagi, Olivia Dumitrescu, Laura Fredrickson, Elba Garc\'ia-Failde, Omar Kidwai, Reinier Kramer, Motohico Mulase, Andy Neitzke, Nikita Nikolaev, Sean Lawton, Laura Schaposnik, Artur Sowa, Jacek Szmigielski, Kaori Tanaka, and Richard Wentworth for useful discussions. The authors are grateful to the organizers of the October 2023 Banff Workshop on Complex Lagrangians, Mirror Symmetry, and Quantization. The writing of this manuscript was finalized during this stimulating workshop. The second-named author is partially supported by a Natural Sciences and Engineering Research Council of Canada (NSERC) Discovery Grant, a Canadian Tri- Agency New Frontiers in Research (NFRF) Exploration Stream Grant, and a Pacific Institute for the Mathematical Sciences (PIMS) Collaborative Research Group (CRG) Award (for which this manuscript is PIMS Report Number PIMS-20240112-CRG34). During the research and preparation of this article, the first-named author held a doctoral fellowship created from the second-named author's NFRF Exploration grant as well as a Teacher-Scholar Doctoral Fellowship from the University of Saskatchewan.

\section{Higgs bundles}



\subsection{Hitchin equations}

Named by Simpson \cite{Simpson92}, the objects called ``Higgs bundles'' were first introduced by Hitchin in \cite{Hitchin87, Hitchin86} as solutions of the self-dual dimensionally-reduced Yang-Mills equations, called the \emph{Hitchin equations}, on a compact Riemann surface $X$.  They play a key role in mathematical physics as a bridge between gauge theory and integrable systems.  Our interest in studying Higgs bundles comes from their relation to spectral curves.  A Higgs bundle is a holomorphic vector bundle with a $K$-valued endomorphism.  Locally, this endomorphism looks like a matrix, so we can look at its characteristic equation.  This generates a spectral curve from a Higgs bundle.

\begin{dfn}\label{DfnHitchinEq}
    Let $\mathcal{E}$ be a Hermitian vector bundle over a Riemann surface $X$ with unitary connection $A$, and $\phi: \mathcal{E}\rightarrow \mathcal{E}\otimes K$ be a smooth linear map.  The \textbf{Hitchin equations} are
    \begin{align*}
        F_{0}(A) + \phi\wedge\phi^{*} &= 0\\
        \nonumber \overline{\partial}_{A}\phi &= 0,
    \end{align*}
    where $F_{0}(A)$ is the trace-free curvature of $A$, and $\overline{\partial}_{A}:\mathcal{C}^{\infty}(\mathcal{E})\to \Omega^{0,1}(\mathcal{E})$ is the holomorphic structure on $\mathcal{E}$ obtained by taking the $(0,1)$-part of $A$.
\end{dfn}

\begin{dfn}
(cf. \cite{Hitchin87}) A \textbf{Higgs bundle} is a pair $(\mathcal{E},\phi)$, where $\mathcal{E}$ is a holomorphic vector bundle on a Riemann surface $X$ and $\phi: \mathcal{E}\rightarrow \mathcal{E}\otimes K$ is a global holomorphic map. The map $\phi$ is called a \textbf{Higgs field}.  The rank and degree of $(\mathcal{E},\phi)$ are the rank and degree of $\mathcal{E}$ respectively.
\end{dfn}

There is a striking similarity between a Higgs bundle $(\mathcal{E},\phi)$ and a solution $(A,\phi)$ of the Hitchin equations.  A natural question is to ask which Higgs bundles correspond to solutions of the Hitchin equations.  The answer is precisely the \emph{stable} Higgs bundles.

\begin{dfn}
    Let $\mathcal{E}$ be a holomorphic vector bundle on a Riemann surface $X$.  The \textbf{slope} of $\mathcal{E}$ is
    \begin{equation}
        \mu(\mathcal{E}) \coloneqq \frac{\deg(\mathcal{E})}{\mathrm{rk}(\mathcal{E})}.
    \end{equation}
\end{dfn}

\begin{dfn}
    A Higgs bundle $(\mathcal{E},\phi)$ is \textbf{stable} if
    \begin{equation}
        \mu(\mathcal{U}) < \mu(\mathcal{E})
    \end{equation}
    for all subbundles $0 \subsetneq \mathcal{U} \subsetneq \mathcal{E}$ satisfying $\ \phi(\mathcal{U})\subseteq \mathcal{U}\otimes K$, and \textbf{semi-stable} if equality is permitted in the slope condition.
\end{dfn}

The equivalence between stable Higgs bundles and solutions to the Hitchin equations is an instance of the Hitchin-Kobayashi theorem (Theorem 4.3 of \cite{Hitchin87}).\\


Before we can make sense of the relationship between Higgs bundles and solutions to the Hitchin equations, we first want to know what kinds of Riemann surfaces have stable Higgs bundles.  For genus $0$ Riemann surfaces, i.e. $X = \mathbb{P}^{1}$, we have the following result:

\begin{prop}
All rank $r \geq 2$ Higgs bundles $(\mathcal{E},\phi)$ on $\mathbb{P}^{1}$ are unstable.
\end{prop}

\begin{proof}
On $\mathbb{P}^{1}$ we have that $K = \mathcal{O}(-2)$. By the Birkhoff-Grothendieck Theorem,
\begin{equation*}
    \mathcal{E}=\oplus_{i=1}^{r}\mathcal{O}(a_{i}),
\end{equation*}
for some $a_{i}$.\\

The Higgs bundle $\phi$ is given by a matrix with components
\begin{equation}
    (\phi)_{ij}=\phi_{ij}:\mathcal{O}(a_i)\rightarrow\mathcal{O}(a_j)\otimes\mathcal{O}(-2).
\end{equation}
Let $a\coloneqq max_{i}\{a_i\}$.  For every $i$ we have
\begin{equation}
    H^0\left(\mathcal{O}(a)^*\otimes\mathcal{O}(a_i)\otimes\mathcal{O}(-2)\right)=H^0(\mathcal{O}(a_i-a-2)).
\end{equation}
By the choice of $a$ we have
\begin{align*}
    a \geq a_i\ &\Rightarrow a_i - a \leq 0\\
    &\Rightarrow a_i - a -2 \leq -2,
\end{align*}
and so
\begin{equation}
    \dim H^0\left(\mathcal{O}(a)^*\otimes\mathcal{O}(a_i)\otimes\mathcal{O}(-2)\right) = 0
\end{equation}
for all $i$.  This means that the column corresponding to $a$ is a column of zeros and thus
\begin{equation}
    \phi: \mathcal{O}(a)\rightarrow \mathcal{O}(a_i)\otimes\mathcal{O}(-2)
\end{equation}
is the zero map.  This shows that $\mathcal{O}(a)$ is a $\phi$-invariant subbundle of $\mathcal{E}$ with $\deg(\mathcal{O}(a))=a$ and $\mbox{rk}(\mathcal{O}(a))=1$.  Comparing the slopes yields

\begin{align*}
    \mu(\mathcal{E}) &= \frac{\sum_{i=1}^{r}a_i}{r}\\
    &\leq \frac{\sum_{i=1}^{r}a}{r}\\
    &= \frac{ra}{r}\\
    &= a\\
    &= \mu(\mathcal{O}(a))
\end{align*}

\noindent meaning that $(\mathcal{E},\phi)$ is not stable.
\end{proof}

A similar result can be proved for genus $1$ Riemann surfaces.  This means the only Riemann surfaces that admit stable Higgs bundles have genus $g \geq 2$.  The simplest examples of such Higgs bundles are given when $\phi = 0$. In this setting, the Hitchin equations reduce to $F(A) = 0$, whose solutions are flat unitary connections.  These objects are already well studied, and a famous result of Narasimhan-Seshadri relates flat unitary connections to stable holomorphic bundles \cite{NarasimhanSeshadri65}.  In particular, the moduli space of stable holomorphic bundles has positive dimension $3g - 3$ when the rank is $2$ and $g > 1$.  The Hitchin-Kobayashi correspondence can be regarded as a generalization of Narasimhan-Seshadri to the non-unitary case.  The non-unitary case arises precisely when $\phi \neq 0$.  An example in rank $2$ of a stable Higgs bundle with $\phi \neq 0$ is given by $\mathcal{E} = K\oplus\mathcal{O}$ with Higgs field given by
\begin{equation}
     \phi = 
    \begin{bmatrix}
    0 & \alpha\\
    1 & 0
    \end{bmatrix}.
\end{equation}
The $1$ can be interpreted as the identity endomorphism as the bottom left entry of $\phi$ is a map $$K\to \mathcal{O}\otimes K = K.$$  The section $\alpha$ is a non-zero element of $H^{0}(X,\mathcal{O}\otimes K\otimes K) = H^{0}(X,K^{2})$, called a \emph{quadratic differential}.  The form of $\phi$ prevents the existence of a proper invariant subbundle $\mathcal{U}$, and therefore the Higgs bundle is automatically stable.  Note that $\alpha \neq 0$ necessitates that $\deg K > 0$, which forces $g$ to be at least $2$.\\

The dimension of $H^{0}(X,K^{2})$ over $\mathbb{C}$ is $3g - 3$, which can be computed via Riemann-Roch and Serre duality.  This entails that this example generates a $(3g - 3)$-dimensional family of stable Higgs bundles, none of which are equivalent to one another as $-\alpha$ is the determinant of the Higgs field and so no two choices of $\alpha$ give Higgs bundles that are isomorphic under change of basis in $\mathcal{E}$.  This indicates that for genus $2$ or larger, the moduli space of rank $2$ Higgs bundles is at least $(3g - 3)$-dimensional. In fact, it is $(6g - 6)$-dimensional.  (This is, of course, twice the dimension of the moduli space of stable bundles, which we will revisit in calculations below.)

\subsubsection{Twisted Higgs bundles}
Riemann surfaces of genus $0$ and $1$ are of interest across of variety of problems, with $\mathbb{P}^{1}$ playing an important role later on in this thesis.  While they do not admit stable Higgs bundles, we can modify the definition of a Higgs bundle to allow for stable bundles.  There are two natural ways of doing this.  We can drop the holomorphic condition on the Higgs field and allow our Higgs bundle to have a Higgs fields with poles.  This leads to \emph{meromorphic Higgs bundles}, or, with more initial data, \emph{parabolic Higgs bundles}.  More generally, we can replace $K$ with another holomorphic line bundle $\mathcal{L}$ on $X$, producing \emph{twisted Higgs bundles}.  More formally:

\begin{dfn}
    (cf. \cite{DumitrescuMulase14b, RayanSchaposnik2020}) Let $D$ be a divisor on $X$.  A \textbf{meromorphic Higgs bundle} with poles at $D$ is a pair $(\mathcal{E},\phi)$, where $\mathcal{E}$ is a holomorphic vector bundle on a Riemann surface $X$ and $\phi: \mathcal{E}\rightarrow \mathcal{E}\otimes K(D)$, where $K(D) = K\otimes D$.
\end{dfn}

\begin{dfn}
    (cf. \cite{Nitsure91, Rayan13}) Let $\mathcal{L}$ be a holomorphic line bundle on a Riemann surface $X$. An \textbf{$\mathcal{L}$-twisted Higgs bundle} on $X$ is a pair $(\mathcal{E},\phi)$, where $\mathcal{E}$ is a holomorphic vector bundle,  and $\phi: \mathcal{E}\rightarrow \mathcal{E}\otimes \mathcal L$.
\end{dfn}

\begin{remark}
    \emph{When working with $\mathcal{L}$-twisted Higgs bundles, we will only be considering the case where $\deg\mathcal{L} > \deg K$.}
\end{remark}


Such Higgs bundles come up in the study of generalized complex geometry as co-Higgs bundles ($\mathcal{L} = K^{*}$) \cite{Rayan14}, in the study of quiver varieties \cite{RayanSundbo18,RayanSundbo19}, in the study of monodromy of the Hitchin map \cite{BaragliaSchaposnik15}, and in the study of representations of fundamental groups of compact K\"ahler manifolds \cite{GarciaRamanan00}.\\


A natural question to ask is if there is an analogous set of Hitchin equations (Definition \ref{DfnHitchinEq}) associated to an $\mathcal{L}$-twisted Higgs bundle.  We can suitably modify the Hitchin equations in the following way. The second equation,
\begin{equation*}
    \overline{\partial}_{A}\phi = 0
\end{equation*}
says that the Higgs field needs to be holomorphic with respect to the connection $A$, a condition that can be considered with no necessary change.  The first equation,
\begin{equation*}
    F_{0}(A) + \phi\wedge\phi^{*} = 0,
\end{equation*}
makes sense because both summands are endomorphism-valued two-forms.  In the twisted setting, $\phi$ is a section of $End(\mathcal{E})\otimes \mathcal{L}$, so in order to make sense of the first equation, we need to choose a section $s\in H^{0}(K\otimes \mathcal{L}^{*})$, with which to multiply $\phi$.  Like this, we can modify the first equation as
\begin{align*}
    F_{0}(A) + s\phi\wedge(s\phi)^{*} = F_{0}(A) + |s|^{2}\phi\wedge\phi^{*} = 0.
\end{align*}

We can also easily modify our notion of stability to this setting.

\begin{dfn}
    An $\mathcal{L}$-twisted Higgs bundle $(\mathcal{E},\phi)$ is \textbf{stable} if
    \begin{equation}
        \mu(\mathcal{U}) < \mu(\mathcal{E})
    \end{equation}
    for all subbundles $0 \subsetneq \mathcal{U} \subsetneq \mathcal{E}$ satisfying $\ \phi(\mathcal{U})\subseteq \mathcal{U}\otimes \mathcal{L}$, and \textbf{semi-stable} if equality is permitted in the slope condition.
\end{dfn}


\subsection{Moduli space of Higgs bundles}

For this section we follow the treatment in \cite{Rayan18}.  The correspondence between solutions to the Hitchin equations and ordinary Higgs bundles descends to the level of moduli spaces, where we have an equivalence between two moduli spaces.  The space of solutions yields a gauge-theoretic moduli space given by the space of solutions $(A, \phi)$ taken up to gauge equivalence.  It has the structure of a smooth, non-compact manifold, which can be interpreted as an infinite-dimensional hyperk\"ahler quotient.  This endows the moduli space with a hyperk\"ahler structure.  On the Higgs bundle side, the algebro-geometric moduli space is formed by quotienting the space of stable pairs $(\mathcal{E}, \phi)$ by the conjugation action of holomorphic automorphisms of $\mathcal{E}$.  This quotient has the structure of a non-singular, quasi-projective variety and can be interpreted as a geometric-invariant theory quotient, with stability condition given by the stability for Higgs bundles as above.  This equivalence extends to the meromorphic and $\mathcal{L}$-twisted setting; however, we lose some properties in the interim, such as the hyperk\"alher structure.\\

The main object of interest for us is the Higgs bundle moduli space, although when necessary, we will appeal to this correspondence to benefit from properties of the gauge-theoretic interpretation.  We will start by considering the larger class of $\mathcal{L}$-twisted Higgs bundles, and study some properties of the $\mathcal{L}=K$ case in a later subsection.  At a set-theoretic level, we define the moduli space of $\mathcal{L}$-twisted Higgs bundles in the following way:
\begin{dfn}
    Fix integers $r>0, d$ coprime.  The \textbf{moduli space of rank $r$, degree $d$ $\mathcal{L}$-twisted Higgs bundles},  $\mathcal{M}_{X}^{\mathcal{L}}(r,d)$, is defined by the quotient
    \begin{equation}
        \mathcal{M}_{X}^{\mathcal{L}}(r,d) = \frac{\{\mbox{stable rank } r, \textrm{degree } d\ \mathcal{L}\textrm{-twisted Higgs bundles } (\mathcal{E},\phi)\}}{\sim},
    \end{equation}
    where the equivalence relation is given by conjugation: $(\mathcal{E},\phi)\sim(\mathcal{E}',\phi')$ iff there exists an invertible bundle map $\psi:\mathcal{E}\to\mathcal{E}'$ such that $\phi' = \psi^{-1}\phi\psi$.
\end{dfn}

\begin{remark}
\emph{The coprime condition guarantees that $\mathcal{M}^{\mathcal{L}}_{X}(r,d)$ is a smooth manifold.}
\end{remark}

\begin{remark*}
    \emph{We will often simplify this notation by writing $\mathcal{M}_{X}^{\mathcal{L}}$, or when working in the $\mathcal{L}=K$ setting by writing $\mathcal{M}_{X}$.}
\end{remark*}

An important tool for studying $\mathcal{M}_{X}^{\mathcal{L}}(r,d)$ is the \emph{Hitchin map}:
\begin{dfn}
    The \textbf{Hitchin map}
    \begin{equation}
        \mathcal{H}: \mathcal{M}_{X}^{\mathcal{L}}(r,d)\to\mathcal{B} = \bigoplus_{i=1}^{r} H^{0}(X,\mathcal{L}^{i})
    \end{equation}
    sends an isomorphism class of a Higgs bundle $(\mathcal{E},\phi)$ to the $r$-tuple of coefficients of the characteristic polynomial of $\phi$.
    The affine space $\mathcal{B}$ is called the \textbf{Hitchin base}.
\end{dfn}

\begin{dfn}\label{DfnNilpotentCone}
    The \textbf{nilpotent cone} is the fiber of the Hitchin map above $0\in \mathcal{B}$, i.e. $\mathcal{H}^{-1}(0)$.
\end{dfn}

Our first step in studying $\mathcal{M}_{X}^{\mathcal{L}}$ is to know its dimension.  A result of Hitchin (in the $\mathcal{L}=K$ case) and Nitsure (in the general case) tells us that the dimension of the moduli space is given by:
\begin{prop}
(\cite{Hitchin87,Nitsure91})
    \begin{itemize}
        \item For $\deg\mathcal{L} > \deg K$:\ \  $\dim\mathcal{M}_{X}^{\mathcal{L}}(r,d) = r^{2}\deg\mathcal{L} + 1$.
        \item For $\mathcal{L} = K$:\ \ $\dim\mathcal{M}_{X}(r,d) = r^{2}(2g-2) + 2$.
    \end{itemize}
\end{prop}
\noindent In a later subsection we will argue briefly why this is true for the $\mathcal{L} = K$ case, and we will prove the result for the $\deg\mathcal{L} > \deg K$ case in \emph{Section 4} when studying the deformation theory of the $\mathcal{L}$-twisted moduli space.

\subsubsection{Spectral correspondence}\label{SubsectionSpectralCorrespondence}
We want to understand now how $\mathcal{M}_{X}^{\mathcal{L}}$ looks.  We will do this by studying the fibres of the Hitchin map.  A result of Hitchin \cite{Hitchin87} and Nitsure \cite{Nitsure91} shows that the Hitchin map is proper, meaning that $\mathcal{M}_{X}^{\mathcal{L}}$ is fibred by compact subvarieties, called \emph{Hitchin fibres}, which can be shown to be tori.\\

Given a line bundle $\mathcal{L}$ on a Riemann surface $X$, the \emph{tautological section} $\eta$ is defined as follows.  Let $(x,y)$ be local coordinates on $\mathrm{Tot}(\mathcal{L})$, where $x$ is the base coordinate and $y$ is the fiber coordinate of $\mathrm{Tot}(\mathcal{L})$, with projection $\pi: \mathcal{L} \to X$.  The pullback bundle $\pi^{*}\mathcal{L}$ has attached to a point $(x,y(x))$ a copy of the fibre $\mathcal{L}_{x}$.  This bundle has a natural section, $\eta$, defined by $\eta(x,y(x)) = y(x) \in (\pi^{*}\mathcal{L})_{y} = \mathcal{L}_{x}$.

\begin{dfn}\label{DefHiggsSC}\label{DfnSpecCurve}
    Let $a = (a_{1},\dots,a_{r}) \in \mathcal{B}$.  The \textbf{spectral curve} $S_{a}$ associated to $a$ is the zero locus of the polynomial
    \begin{equation}\label{EqnSpectralCurve}
        \det(\eta - \pi^{*}\phi) = \eta^{r} + a_{1}\eta^{r-1} + \dots + a_{r}.
    \end{equation}
\end{dfn}

\begin{remark*}
    \emph{When $(\mathcal{E},\phi)$ is a Higgs bundle such that $\mathcal{H}([\mathcal{E},\phi]) = a$, then we will say that the spectral curve is associated to $(\mathcal{E}, \phi)$.}
\end{remark*}

It can be shown that the Hitchin fibres are precisely Jacobians of the respective spectral curves, by identifying eigenspaces of a given $\phi$ with holomorphic line bundles on the spectral curve.  In other words, the Hitchin fibration is globally a torus fibration and, in fact, is a completely integrable Hamiltonian system whose Poisson structure is induced from the open dense embedding of the cotangent bundle of the moduli space of bundles and whose Hamiltonians are the real and imaginary components of the Hitchin map \cite{Hitchin86}.  This integrable system has attracted immense interest in both mathematics and physics, as every classical integrable system is believed to be an instance, limit, or modification of the Hitchin system (\emph{cf.} for instance \cite{DonagiMarkman96}).  Rather than do justice to integrability in the framework of Higgs bundles, we will just argue that the fibres are in fact Jacobians of the respective spectral curves.\\

\begin{figure}[htp] 
    \centering
    \includegraphics[width=10cm]{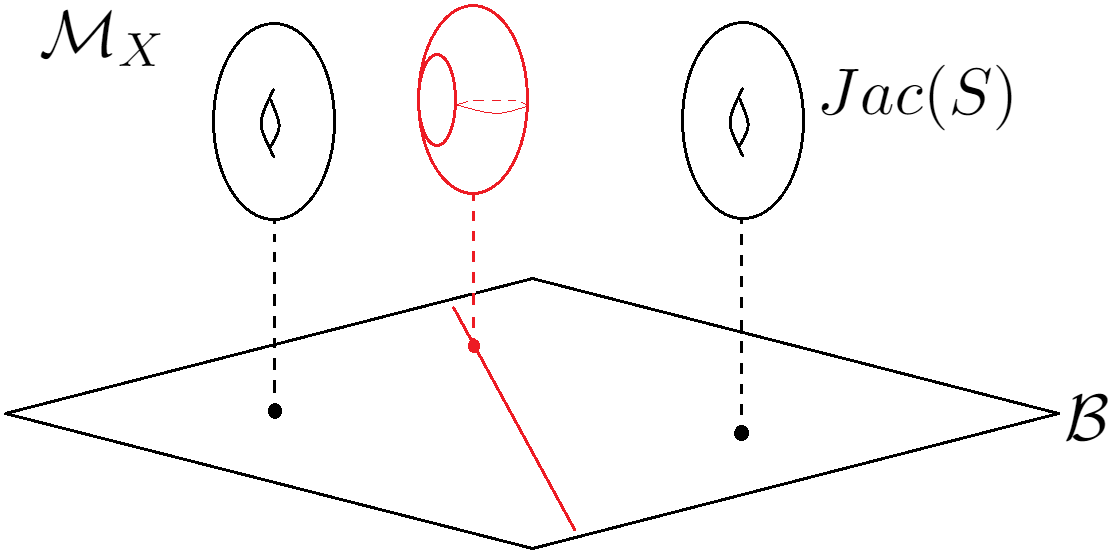}
    \caption{Generic fibres of $\mathcal{M}_{X}^{L}$ are tori.  There is a locus of degenerate singular tori.}
\end{figure}

For a generic choice of $a\in\mathcal{B}$, the spectral curve $S_{a}$ is an $r:1$ branched cover of $X$ living inside of $\mathrm{Tot}(\mathcal{L})$.  We can restrict the bundle projection map $\pi:\mathrm{Tot}(\mathcal{L})\to X$ to $S_{a}$ to get a map $\pi:S_{a}\to X$.  Let $\mathcal{Q}$ be a line bundle on $S_{a}$.  On this line bundle, the tautological section can be seen as acting by
\begin{align*}
    \eta|_{S_{a}}: & \mathcal{Q} \to \mathcal{Q}\otimes \pi^{*}\mathcal{L}\\
    & s \mapsto s\cdot y.
\end{align*}
The direct image $\mathcal{E} = \pi_{*}\mathcal{Q}$ is a rank $r$ vector bundle over $X$.  The tautological section pushes forward to a linear map $\mathcal{E}\to\mathcal{E}\otimes \mathcal{L}$, an $\mathcal{L}$-twisted Higgs field for $\mathcal{E}$.  Thus, from the data of a line bundle $\mathcal{Q}$ on $S_{a}$, we can produce an $\mathcal{L}$-twisted Higgs bundle on $X$.  To see this in the other direction, the Hitchin map sends a Higgs field $(\mathcal{E}, \phi)$ to an $r$-tuple of spectral data $a\in\mathcal{B}$. From this data, we can produce a spectral curve $S_{a}$, which, by Definition \ref{DefHiggsSC}, is the spectrum of the $\phi$, having $r$ distinct eigenvalues generically over $x\in X$ in correspondence to $S_{a}$ being a branched $r:1$ cover of $X$ (\emph{i.e.} branched points correspond to repeated eigenvalues).  The eigenspaces of $\phi$ are generically $1$-dimensional, and form a line bundle over $S_{a}$.\\

What we have shown thus far is the \emph{spectral correspondence} \cite{BEAUVILLEA1989SCAT, Donagi95, DonagiMarkman96, Hitchin87}: an isomorphism class of holomorphic line bundles $[\mathcal{Q}]$ on $S_{a}$ is equivalent to the data of an isomorphism class of Higgs bundles $[(\mathcal{E}, \phi)]$ on $X$.  It then follows that a generic fiber $\mathcal{H}^{-1}(a)$ is isomorphic to the Jacobian variety of the spectral curve $S_{a}$.  This is not the space of degree 0 line bundles, but rather the degree is given by

\begin{equation*}
    deg{\widetilde{\mathcal{L}}} = d - (1-g_{S_{a}}) + r(1-g).
\end{equation*}

\subsubsection{$\mathbb{C}^{*}$-action}\label{SubsectionC*Action}
Our description up to this point is not enough to fully understand the topology of the moduli space because of the presence of special degenerate fibers.  To continue the investigation, we appeal to the correspondence between the gauge-theoretic and algebro-geometric pictures.  Studying the gauge-theoretic moduli space, we could employ Morse-Bott theory on a symplectic leaf that contains the nilpotent cone (given that the whole of the moduli space is not naturally K\"ahler when $\mathcal L\ncong K$), and justifying that the (rational) cohomology of the leaf coincides with that of the whole moduli space.  Here, we must study the critical points of the height function $f(\mathcal{E}, \phi) = \frac{1}{2}||\phi||^{2}$, the $L^{2}$-norm on the moduli space.  On the algebro-geometric side, we can employ Bia\l{}ynicki-Birula theory \cite{Bialynicki-Birula73} and study the fixed points for the algebraic group action
\begin{align*}
    \lambda.(\mathcal{E},\phi) = (\mathcal{E}, \lambda\cdot\phi)
\end{align*}
of $\mathbb{C}^{*}$.  These two points of view come together by the following facts:
\begin{itemize}
    \item the fixed points of the $\mathbb{C}^{*}$-action are fixed points of $U(1)\subset\mathbb{C}^{*}$
    \begin{equation*}
        \lambda.(\mathcal{E},\phi) = (\mathcal{E}, e^{i\lambda}\cdot\phi)
    \end{equation*}
    \item the height function $f$ is a moment map for the $U(1)$-action, and the fixed points of the $U(1)$-action are critical points of $f$.
\end{itemize}
\noindent What is fortunate about the Bia\l ynicki-Birula approach is that there is no need to appeal to a K\"ahler structure.\\


A theorem of Frankel \cite{Frankel59} tells us that $f$ is a nondegenerate perfect Morse-Bott function, and so the Poinca\'re polynomial of $\mathcal{M}_{X}(r,d)$ is given by
\begin{equation}
    P[\mathcal{M}_{X}(r,d)](t) = \sum_{i\in I}t^{\beta(\mathcal{C}_{i})}P[\mathcal{C}_{i}](t),
\end{equation}
where $\mathcal{C}_{i}$ are the critical subvarieties of $f$, and $\beta(\mathcal{C}_{i})$ is the Morse index of $\mathcal{C}_{i}$, \emph{i.e.} the rank of the subbundle of the normal bundle on which the Hessian of $f$ is negative definite.  This tells us that the topology of the moduli space is contained in this fixed locus.\\

We now focus our attention on the set of fixed points of the $U(1)$-action, which we denote by $\mathcal{M}_{X}(r,d)^{U(1)}$.  Stable Higgs bundles are fixed \emph{iff} there exists an automorphism $A_{\lambda}$ of $\mathcal{E}$ such that
\begin{equation}\label{EqnFixedPointAuto}
    A_{\lambda}\phi A_{\lambda}^{-1} = e^{i\lambda}\phi
\end{equation}
for all $\lambda\in [0, 2\pi)$, \emph{i.e} there is a change of basis that undoes the $U(1)$-action.\\

We would like to develop a better description of the fixed points.  Let $(\mathcal{E}, \phi) \in \mathcal{M}_{X}(r,d)^{U(1)}$ with $A_{\lambda}$ the one-parameter family of automorphisms that satisfy \eqref{EqnFixedPointAuto}.  There is a limiting endomorphism $\Lambda$ that generates the family $A_{\lambda}$ infinitesimally,
\begin{equation*}
    \Lambda \coloneqq D_{\lambda}(A_{\lambda})|_{\lambda=0},
\end{equation*}
where $D_{\lambda}$ is a suitably-defined derivative.  This limiting endomorphism interacts with the Higgs field in the following way:

\begin{lem}\label{LemHoloChain}
    $[\Lambda, \phi] = i\phi$.
\end{lem}

\begin{proof}
    We start with \eqref{EqnFixedPointAuto} and apply $D_{\lambda}(\cdot)|_{\lambda=0}$ to both sides.\\
    On the right-hand side, we have
    \begin{equation*}
        D_{\lambda}(e^{i\lambda}\phi)|_{\lambda=0} = ie^{i\lambda}\phi|_{\lambda=0} = i\phi.
    \end{equation*}
    On the left-hand side, we have
    \begin{align*}
        D_{\lambda}(A_{\lambda}\phi A_{\lambda}^{-1}) &= D_{\lambda}(A_{\lambda})\phi A_{\lambda}^{-1}|_{\lambda=0} + A_{\lambda}\phi D_{\lambda}(A_{\lambda}^{-1})|_{\lambda=0}\\
        &= D_{\lambda}(A_{\lambda})\phi A_{\lambda}^{-1}|_{\lambda=0} + A_{\lambda}\phi (-1)D_{\lambda}(A_{\lambda})A_{\lambda}^{-2}|_{\lambda=0}\\
        &= \Lambda\phi - \phi\Lambda\\
        &= [\Lambda, \phi],
    \end{align*}
    where we use the definition of $\Lambda$, and the fact that $A_{0}$ is the identity map.
\end{proof}

Returning to the gauge-theoretic viewpoint, for a pair $(A, \phi)$ that satisfy the Hitchin equations, we have that the $\overline{\partial}$-operator induced by $A$, denoted $\overline{\partial}_{A}$, satisfies
\begin{equation*}
    \overline{\partial}_{A}\Lambda = 0.
\end{equation*}
In this way, we have that $\mathcal{E}$ decomposes into a direct sum of eigenspaces $\mathcal{B}_{1}, \dots, \mathcal{B}_{n}$ of $\Lambda$,
\begin{equation*}
    \mathcal{E} = \oplus_{k=1}^{n}\mathcal{B}_{k}.
\end{equation*}
These eigenspaces are in fact holomorphic subbundles of $\mathcal{E}$, and the corresponding eigenvalues $s_{1},\dots, s_{n}$ of $\Lambda$ are global holomorphic functions on $X$.  Applying both sides of Lemma \ref{LemHoloChain} to some $\mathcal{B}_{k}$, we find that
\begin{equation*}
    \Lambda(\phi\mathcal{B}_{k}) = (s_{k}+1)(\phi\mathcal{B}_{k}).
\end{equation*}
The image of $\mathcal{B}_{k}$ under $\phi$ is then a subbundle of the eigenbundle for eigenvalue $s_{k}+i$.  This means we can re-index (as needed) and group the eigenspaces into sequences with eigenvalues $s_{k}, s_{k}+i, s_{k}+2i, \dots$, terminating when the image of an eigenbundle under $\phi$ is zero (\emph{i.e.} when we have reached the last eigenbundle).  There cannot be multiple, disconnected sequences for a fixed point, as that would violate the stability condition.  This means that for a fixed point of the $U(1)$-action, $(\mathcal{E}, \phi)\in \mathcal{M}_{X}(r,d)^{U(1)}$, there is a number $n$ such that $\mathcal{E} = \oplus_{k=1}^{n}\mathcal{B}_{k}$, and
\begin{equation*}
    \mathcal{B}_{1} \xrightarrow{\phi_{1}} \mathcal{B}_{2}\otimes\mathcal{L} \xrightarrow{\phi_{2}} \dots \xrightarrow{\phi_{n-1}} \mathcal{B}_{n}\otimes\mathcal{L}^{\otimes n-1} \xrightarrow{\phi_{n}} 0,
\end{equation*}
where $\phi_{k} \coloneqq \phi|_{\mathcal{B}_{k}}$ is not identically zero for $k < n$.  A Higgs bundle satisfying this description is called a \emph{holomorphic chain}, \emph{cf.} \cite{AlvarezGarcia01, AlvarezGarciaSchmitt06, BradlowGarciaGothen04, GarciaHeinlothSchmitt14}.  In the case when $L = K$, these can be regarded as complex variations of Hodge structures \cite{Simpson90}.

With this description, we can write a fixed point in a basis of sections where $\phi$ looks like:
\begin{equation}\label{EqnFixedPointForm}
    \phi = \begin{pmatrix}
    0 & 0 & \cdots & 0 & 0\\
    \phi_{1} & 0 & \cdots & 0 & 0\\
    0 & \phi_{2} & \cdots & 0 & 0\\
     &  & \ddots &  & \\
    0 & 0 & \cdots & \phi_{n-1} & 0
    \end{pmatrix}.
\end{equation}
This local matrix description of $\phi$ is nilpotent, and since every fixed  point can be written in this form, they all live in the nilpotent cone (Definition \ref{DfnNilpotentCone}).\\

Notably, not all points within the nilpotent cone are fixed points of the action.  Only those that can be written in the form of \eqref{EqnFixedPointForm} are fixed.  Regardless, the topological information of $\mathcal{M}_{X}(r,d)$ is contained within the nilpotent cone.\\

We can also think of the fixed points of this action as $\mathcal{L}$-twisted representations of $A$-type quivers, with lengths and labels determined by partitions of $r$ and $d$:
\begin{equation*}
    \bullet_{r_{1},d_{1}}\to \bullet_{r_{2}, d_{2}}\to \cdots \to \bullet_{r_{n}, d_{n}}.
\end{equation*}
Such a representation is called a \emph{quiver bundle}, \emph{cf.} \cite{Gothen95, GothenKing05, Rayan16, RayanSundbo18, Schmitt05}.


\subsubsection{$\mathcal{L} = K$ moduli space}
We wish to restrict now to the $\mathcal{L} = K$ case and take an opportunity to look specifically at the ordinary Higgs bundle moduli space, $\mathcal{M}_{X}(r,d)$.  We will start by looking at the Hitchin base $\mathcal{B}$, aiming to show that
\begin{equation}
	\dim\mathcal{B}=r^2(g-1)+1.
\end{equation}

We begin by applying Serre duality to the $h^1(K^n)$ terms.
\begin{align*}
    h^1(K^n) &= h^0(K\otimes(K^n)^*)\\
    &= h^0(K\otimes(K^*)^n)\\
    &= h^0((K^*)^{n-1})\\
    &= 0
\end{align*}
Because $\deg(K^{n}) = (2g-2)n$, we know that $(K^*)^{n-1}$ has degree $(2-2g)(n-1) < 0$, meaning that $h^1(K^n)$ vanishes.\\

Applying the Riemann-Roch theorem for $n>1$:
\begin{align*}
    h^{0}(K^{n}) &= \deg K^n + rk K^n(1-g)\\
    &= (2g-2)n + (1-g)\\
    &= (2n)g - 2n + 1 - g\\
    &= (2n-1)g - (2n-1)\\
    &= (g-1)(2n-1).
\end{align*}

Putting it all together we get that
\begin{align*}
    h^0(K^n) &= 
    \begin{dcases}
        g & n = 1 \\
        (g-1)(2n-1) & n\geq 2 \\
    \end{dcases}
\end{align*}

Now that we know the dimension of each homology group, we can sum them together to get the dimension of the Hitchin base.

\begin{align*}
    \dim\mathcal{B} &= \sum_{i=1}^{r}h^0(K^i)\\
    &= \sum_{i=2}^{r}h^0(K^i) + h^0(K)\\
    &= \sum_{i=2}^{r}(g-1)(2i-1) + g\\
    &= (g-1)\sum_{i=2}^{r}(2i-1) + g\\
    \\
    &= (g-1)\left(2\sum_{i=2}^{r}i - \sum_{i=2}^{r}1\right) + g\\
    \\
    &= (g-1)\left(2(\frac{r(r+1)}{2}-1) - r + 1\right) + g\\
    \\
    &= (g-1)(r^2+r-2-r+1) + g\\
    &= (g-1)(r^2-1) + g\\
    &= r^2(g-1) - (g-1) + g\\
    &= r^2(g-1) + 1.
\end{align*}

\vspace{10pt}

Using the deformation theory of sheaves, we can argue that the tangent space to the moduli space of stable bundles at any point $\mathcal{E}$ is isomorphic to the cohomology $H^{1}(X,\mbox{End}(\mathcal{E}))$, which by Riemann-Roch also has dimension $r^{2}(g-1) + 1$.  By Serre duality, the cotangent space is, therefore, $H^{0}(X,\mbox{End}(\mathcal{E})\otimes K)$, which is the space of possible Higgs fields for $\mathcal{E}$.  In other words, $\mathcal{M}_{X}(r,d)$ contains the cotangent bundle to the moduli space of stable bundles.  This containment is open dense, and so the dimension of $\mathcal{M}_{X}(r,d)$ is $2r^{2}(g-1) + 2$.\\

From this, it turns out that the dimension of $\mathcal{M}_{X}(r,d)$ is twice the dimension of both the Hitchin base and the moduli space of stable bundles.  These two different fibrations are related by what is called ``hyperk\"ahler rotation''.

\begin{ex}\label{ExHitchinSection1}
\emph{Returning to the example $\mathcal{E} = K\oplus\mathcal{O}$ with Higgs field given globally by the matrix
\begin{equation}
     \phi = 
    \begin{bmatrix}
    0 & \alpha\\
    1 & 0
    \end{bmatrix},
\end{equation}
which acts by multiplication on sections of $\mathcal{E}$.  The Hitchin map here sends $\phi$ to $\eta^{2} - \alpha\in H^{0}(X,K^{2})$.  These Higgs bundles form the \emph{Hitchin section}, a locus of Higgs bundles that intersects the Hitchin fibres at exactly one point each, as for every element of $\mathcal{B}$ there is only one Higgs field with the above form.}
\end{ex}

\begin{remark}\label{ExHitchinSection2}
\emph{It is possible to reformulate Example \ref{ExHitchinSection1} with a different Higgs bundle, $\mathcal{E} = K^{\frac{1}{2}}\oplus K^{-\frac{1}{2}}$ with Higgs field
\begin{equation}
     \phi = 
    \begin{bmatrix}
    0 & \alpha\\
    1 & 0
    \end{bmatrix},
\end{equation}
where $K^{\frac{1}{2}}$ is a choice of holomorphic square root of $K$  (there are $2^{2g}$ such choices).  The Higgs field has similar properties to the example above, in particular $\eta - \det\phi \in H^{0}(X,K^{2})$.  The main difference between these examples is the degree of the Higgs bundle. The degree of $\mathcal{E} = K\oplus\mathcal{O}$ is $2g - 2$, while the degree of $\mathcal{E} = K^{\frac{1}{2}}\oplus K^{-\frac{1}{2}}$ is $0$.}
\end{remark}

\subsubsection{$SL(r,\mathbb{C})$-Higgs bundles}
Until now, we have been considering the situation where the structure group for $\mathcal{E}$ is $GL(r,\mathbb{C})$.  There is often a preference to fix the determinant of the Higgs bundle, with the convention being to either let $\det\mathcal{E}=\mathcal{P}$ for some fixed line bundle $\mathcal{P}$, or $\det\mathcal{E} = \mathcal{O}_{X}$ specifically.  For our purposes, we will consider the latter.  Fixing the determinant takes us from the $GL(r,\mathbb{C})$ setting to the $SL(r,\mathbb{C})$ setting.

\begin{dfn}
    Let $\mathcal{L}$ be a holomorphic line bundle on a Riemann surface $X$. An \textbf{$SL(r,\mathbb{C})$-Higgs bundle} on $X$ is a pair $(\mathcal{E},\phi)$, where $\mathcal{E}$ is a holomorphic vector bundle with $\det\mathcal{E} = \mathcal{O}_{X}$,  and $\phi: \mathcal{E}\rightarrow \mathcal{E}\otimes \mathcal L$ such that $\emph{tr}(\phi) = 0$.
\end{dfn}

The moduli space of $SL(r,\mathbb{C})$-Higgs bundles $\mathcal{M}_{X}^{\mathcal{L},0}$ is a subvariety of $\mathcal{M}_{X}^{\mathcal{L}}$.  The restricted Hitchin map on $\mathcal{M}_{X}^{\mathcal{L},0}(r,d)$ has codomain $\mathcal{B}^{0} \coloneqq \bigoplus_{i=2}^{r} H^{0}(X,\mathcal{L}^{i})$ 
 since $\text{tr}\,\phi$ belongs to $H^0(X,\mathcal{L})$. In particular, the nilpotent cones of $\mathcal{M}_{X}^{\mathcal{L},0}$ and $\mathcal{M}_{X}^{\mathcal{L}}$ are coincident, as Higgs bundles in the nilpotent cone have, by definition, trace-free Higgs fields even in the $GL(r,\mathbb C)$ case.  The dimension of $\mathcal{M}_{X}^{\mathcal{L},0}$ is $\deg\mathcal{L}(r^2 - 1)$.


\section{Topological recursion}


\subsection{The prime form and fundamental differentials}
In order to define and study topological recursion in later sections, we will need to make use of certain fundamental differentials on a Riemann surface and their properties.  Such objects are studied in depth in \cite{Fay73, Mumford07}.\\

Let $X$ be a Riemann surface of genus $g$.  Choose a symplectic basis $\langle A_{1},...,A_{g},B_{1},...,B_{g} \rangle$ for $H_{1}(X,\mathbb{Z})$\footnote{\cite{BouchardEynard17} call a Riemann surface together with such a symplectic basis a \textbf{Torelli surface}.}.  Let $v_{1},\dots,v_{g}$ be a basis of holomorphic differentials, normalized by
\begin{equation}
    \int_{A_{j}}v_{i} = \delta_{ij}.
\end{equation}
With respect to the symplectic basis, the period matrix $\tau$ of $X$ is given by
\begin{equation}
        \tau_{ij} = \int_{B_{j}}v_{i}.
\end{equation}
We can identify the Jacobian of $X$ as $Jac(X) = Pic^{0}(X)$, which is isomorphic to $Pic^{g-1}(X)$.  The \emph{theta divisor} $\Theta$ of $Pic^{g-1}(X)$ is defined by
\begin{equation*}
    \Theta = \{ \mathcal{L}\in Pic^{g-1}(X) | \dim H^{1}(X,\mathcal{L}) > 0 \}.
\end{equation*}
Consider the diagram
\begin{center}
        \begin{tikzpicture}
            \matrix (m) [matrix of math nodes, row sep=2em, column sep=3em]
            { & Jac(X) & \\
            & X\times X &\\
            X & & X\\};
            \path[-stealth]
            (m-2-2) edge node [right]{$\delta$}  (m-1-2)
            (m-2-2) edge node [above] {$\pi_{1}$} (m-3-1)
            (m-2-2) edge node [above] {$\pi_{2}$} (m-3-3);
        \end{tikzpicture}
\end{center}
where $\pi_{j}$ denotes projection onto the $j$th component, and
\begin{equation*}
    \delta: X\times X\ni (p,q) \longmapsto p-q \in Jac(X).
\end{equation*}

\begin{dfn}\label{DfnPrimeForm}
    The \textbf{prime form} $E(z_{1}, z_{2})$ is defined as a holomorphic section
    \begin{equation*}
        E(p,q)\in H^{0}\left( X\times X, \pi_{1}(K)^{-\frac{1}{2}}\otimes \pi_{2}(K)^{-\frac{1}{2}}\otimes \delta^{*}(\Theta)\right)
    \end{equation*}
    where we choose Riemann's spin structure (or the Szeg\"o kernel) $K^{\frac{1}{2}}$, which has a unique global section up to the constant multiplication (\cite{Fay73} Theorem 1.1).
\end{dfn}

The prime form satisfies the following properties:
\begin{itemize}
    \item $E(p,q)$ vanishes to first order along the diagonal $\Delta\in X\times X$, and is otherwise nonzero.
    \item $E(p,q) = - E(q,p)$.
    \item Let $z$ be a local coordinate on $X$.  This means that $dz(p)$ gives a local trivialization of $K$ around $p$.  At a point $q$ near to $p$, $\delta^{*}(\Theta)$ is also trivialized around $(p,q)\in X\times X$ with local expression
    \begin{equation}
        E(z(p), z(q)) = \frac{z(p)-z(q)}{\sqrt{dz(p)}\cdot \sqrt{dz(q)}}\left(1 + O((z(p) - z(q))^{2})\right).
    \end{equation}
\end{itemize}

\begin{dfn}\label{DfnBergmanKernel}
The \textbf{fundamental normalized differential of the second kind} (or \textbf{Bergman kernel}) \begin{equation}
    B(z_{1},z_{2}) = d_{1}d_{2}\log E(z_{1},z_{2}) \in H^{0}(X\times X, \pi_{1}^{*}(K)\otimes\pi_{2}^{*}(K)\otimes\mathcal{O}(2\Delta))
\end{equation}
is the unique bi-linear meromorphic differential on $X\times X$ that satisfies the following:
\begin{itemize}
    \item It is symmetric: $B(z_{1},z_{2}) = B(z_{2},z_{1})$.
    \item It is holomorphic everywhere except for a double pole at $z_{1} = z_{2}$ and can be expanded in a neighbourhood of the diagonal as
        \begin{equation}
            B(z_{1},z_{2}) = \frac{dz_1dz_2}{(z_1-z_2)^2} + O(1)dz_1dz_2.
        \end{equation}
    \item It is normalized on $A$-cycles:
        \begin{equation}
            \oint_{z_{1}\in A_{i}} B(z_{1},z_{2}) = 0
        \end{equation}
        for $i = 1,...,g$.
    \item The integrals along $B$-cycles are given by:
    \begin{equation}
        \oint_{z_{1}\in B_{j}}B(z_{1},z_{2}) = 2\pi iv_{j}(z_{2})
    \end{equation}
    for $i = 1,...,g$.
\end{itemize}
\end{dfn}

\begin{dfn}\label{DfnNormalizedCauchyKernel}
    The \textbf{normalized differential of the third kind} (or \textbf{normalized Cauchy kernel})\\
    \begin{equation}
        \omega^{a-b}(z) = d_{z}\log\frac{E(a,z)}{E(b,z)} = \int_{b}^{a}B(t,z)
    \end{equation}
    is the unique meromorphic differential on $X$ satisfying;
    \begin{itemize}
        \item It is holomorphic except for $z=a$ and $z=b$.
        \item It is normalized on $A$-cycles:
        \begin{equation}
            \oint_{z\in A_{i}} \omega(z) = 0
        \end{equation}
        for $i = 1,...,g$.
        \item It has a simple pole at $z=a$ with residue $1$ and at $z=b$ with residue $-1$.
    \end{itemize}
\end{dfn}


\subsection{Matrix models}

Topological recursion was developed in the context of matrix models and random matrix theory.  A main concern in random matrix theory is the statistical behaviour of the spectrum of a matrix, particularly in a large $N$ limit.  For most ``nice'' cases, the density of eigenvalues converges to a continuous density function, referred to as the \emph{equilibrium measure}.  This equilibrium measure has compact support and takes the form of a complex algebraic function.  This means that to a sufficiently ``nice'' random matrix model, we can associate an algebraic curve $S$.\\

Two classic examples that demonstrate this are the Wigner semi-circle distribution, which is the equilibrium measure for eigenvalues of a Gaussian random matrix given by
\begin{equation}
    \rho(x)dx = \frac{1}{2\pi}\sqrt{4-x^{2}}\chi_{[-2,2]}dx,
\end{equation}
and the Marchenko-Pastur distribution, which describes the asymptotic behaviour of $M\times N$ Gaussian random matrices given by
\begin{equation}
    \rho(x)dx = \frac{N}{2\pi M\sigma^{2}} \frac{\sqrt{\frac{M}{N}\sigma^{4} - \frac{M^{2}}{N^{2}}\sigma^{4} - (x-\sigma^{2})^{2}}}{x} \chi_{[a,b]}dx,
\end{equation}
where $\sigma^{2}$ is the variance, and $a,b$ are the roots of term in the square-root.\\

It is known (\emph{cf.} for instance \cite{Eynard2014}) that understanding the algebraic curve $S$ associated to an equilibrium measure is enough to understand the asymptotic expansion of all expectation values to all orders.  In the large $N$ limit, the expectation value of a multi-resolvent, the joint probability of $n$-eigenvalues, can be expanded as
\begin{equation*}
    \left\langle \left(\text{Tr} \frac{1}{x_{1} - M}\right) \dots \left(\text{Tr} \frac{1}{x_{n} - M}\right) \right\rangle \sim \sum_{g=0}^{\infty} N^{2-2g-n} W_{g,n}(x_{1},...,x_{n}).
\end{equation*}
Understanding the $W_{g,n}$ is what is necessary to compute all correlation functions of the matrix model.  The $W_{g,n}$ are differential forms defined recursively on $2 - 2g - n$, satisfying a relation
\begin{equation*}
    W_{g,n+1} \sim W_{g-1,n+2} + \sum_{\substack{g_1+g_2=g \\ |I\cup J| = n}} W_{g_1, |I|+1}W_{g_2, |J|+1},
\end{equation*}
and depend only on the information of $S$.


\subsection{Eynard-Orantin differentials}

A natural question to ask is ``what happens when we compute these $W_{g,n}$ for a \emph{spectral curve}, an algebraic curve arising as the spectrum of an arbitrary matrix-valued function?''  This defines a family of invariants $\{ W_{g,n} \}$ of the curve called the \emph{Eynard-Orantin differentials}.


\begin{dfn}\label{DfnTRSpectralCurve}
    A \textbf{spectral curve} is a triple $(S,x,y)$ where $S$ is a compact Riemann surface, and $x,y:S\to \mathbb{P}^{1}$.
\end{dfn}

We can view a spectral curve $S$ as an $r:1$ cover of $\mathbb{P}^{1}$ with covering map $x:S \rightarrow \mathbb{P}^{1}$.  Viewing $x$ as the local coordinate on $S$ and $y$ as the fiber coordinate of $T^{*}\mathbb{P}^{1}$, the spectral curve is defined by
\begin{equation}\label{EqnTRSpectralCurve}
    \{ (x,y) : P(x,y) = y^r + \sum_{i=1}^{r}f_{i}(x)y^{r-i} = 0\}.
\end{equation}

\begin{remark*}
    \emph{If we have a rank $r$ Higgs bundle $(\mathcal{E},\phi)$ on $\mathbb{P}^{1}$, we can think of \eqref{EqnTRSpectralCurve} as a local expression of \eqref{EqnSpectralCurve}}
\end{remark*}

In this section, we are interested in spectral curves of the form
\begin{equation}
    P(x,y) = y^2 - f(x) = 0,
\end{equation}
and we consider an affine coordinate $z$ around $p$ on $S$.  The topological recursion lives on

\begin{center}
\begin{tikzpicture}
    \matrix (m) [matrix of math nodes, row sep=3em, column sep=1em]
    { K_{S} &   & \\
      S &  & T^{*}\mathbb{P}^{1} \\
      & \mathbb{P}^{1} & \\ };
    \path[-stealth]
    (m-1-1) edge node {}  (m-2-1)
    (m-2-1) edge node {}  (m-2-3)
            edge node [left] {$x$} (m-3-2)
    (m-2-3) edge node [right] {$\pi$} (m-3-2);
\end{tikzpicture}
\end{center}
where $\pi|_{S} = x$.\\

We will impose an additional condition on the ramification points of our spectral curve.  While it is possible to be general, we will only be interested in this specific case.

\begin{dfn} \label{DfnGoodSpec}
    A \textbf{good spectral curve} over $\mathbb{P}^{1}$ is a spectral curve as above, with only simple ramification points and the additional condition that zeroes of $dx$ and $dy$ do not coincide.
\end{dfn}

For the remainder of this section, we will be exclusively interested in good spectral curves, and will refer to them as spectral curves.\\

The ramification divisor $R$ of a spectral curve consists of points where $f(x)$ has zeroes or poles.  For the spectral curves that we are considering, these are points where zeroes of $dx$ or poles of $x$ are order 2 or greater. Around each ramification point $p\in R$, there is a local involution $\sigma_{p}$ that fixes the ramification point and exchanges nearby points between the sheets (see Figure \ref{FigSpectral}).\\

\begin{figure}[htp] 
    \centering
    \includegraphics[width=10cm]{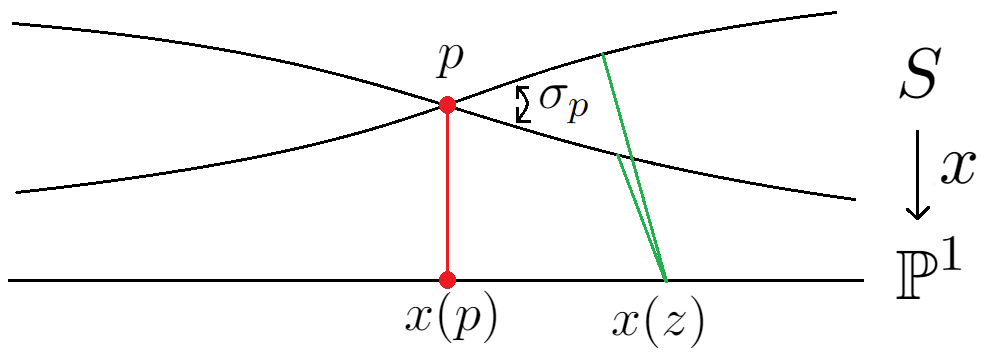}
    \caption{Local picture of a spectral curve $S$ on $\mathbb{P}^{1}$.}
    \label{FigSpectral}
\end{figure}

The Bergman kernel (Definition \ref{DfnBergmanKernel}) plays an important role in the definition of the Eynard-Orantin invariants.  As such, we need to choose a symplectic basis $\langle A_{1},...,A_{\tilde{g}},B_{1},...,B_{\tilde{g}} \rangle$ for $H_{1}(X,\mathbb{Z})$.

\begin{dfn}
Let $p\in R$.  The \textbf{recursion kernel} at $p$ is a meromorphic section of $K_{S}\boxtimes K_{S}^{*}$ defined by
    \begin{equation}\label{EqnRecursionKernel}
        K_{p}(z_0,z) = \frac{\int_{t=\alpha}^{z}B(t,z_0)}{(y(z)-y(\sigma_{p}(z))dx(z)}
    \end{equation}
where $\alpha$ is an arbitrary base point, and $B$ is the Bergman kernel.
\end{dfn}

\begin{remark*}
\emph{The term $\frac{1}{dx(z)}$ is acting as a contraction operation with the vector field $(\frac{dx}{dz})^{-1}\frac{\partial}{\partial z}$.  In this way the ``division" acts by killing terms of the form $dx(z)$ in the numerator.}
\end{remark*}

\begin{dfn}\label{DfnOGTR}
The \textbf{Eynard-Orantin differentials} $W_{g,n}$ are meromorphic sections of the $n$-th exterior tensor product $K_{S}^{\boxtimes n}$, i.e. multi-differentials, defined as follows:\\

The initial conditions of the recursion are given by:
\begin{itemize}
    \item $W_{0,1}$ is a meromorphic $1$-form on $S$.\footnote{The choice of $W_{0,1}$ is related to the geometric problem that one is studying.  In the original formulation \cite{EynardOrantin07b}, $W_{0,1}$ is taken to be $0$.  In many other examples, including the Airy spectral curve, it is taken to be $ydx$.}
    \item $W_{0,2}(z_1,z_2) = B(z_{1},z_{2})$\label{EqnW02}.
\end{itemize}

For all $g,n\in \N$ and $2g - 2 + n \geq 0$, define $W_{g,n}$ recursively by
\begin{align}
    W_{g,n+1}(z_0, \textbf{z}) = \sum_{p\in R}&\mathrm{Res}_{z = p}K_{p}(z_0,z)\Big[ W_{g-1,n+2}(z, \sigma_{p}(z),\textbf{z})\\
    \nonumber &+\sum_{\substack{g_1+g_2=g \\ I\cup J = \textbf{z}}}^{'}W_{g_1, |I|+1}(z,I)W_{g_2, |J|+1}(\sigma_{p}(z), J)\Big] \label{EqnTR}
\end{align}
where the prime signifies summation excluding the cases $(g_1,I)$ or $(g_2,J) = (0,0)$.  

\noindent The terms with $k = 2g +n -1 \geq 2$ are called \textbf{stable differentials}.
\end{dfn}

\vspace{10pt}

To illustrate the general form of the Eynard-Orantin differentials, we compute the first few.\\

$k=1:$
\begin{equation}
    W_{0,2}(z_1,z_2) = B(z_{1},z_{2})
\end{equation}

$k=2:$
\begin{align}
    W_{0,3}(z_{0},z_{1},z_{2}) &= \sum_{p\in R}\mathrm{Res}_{z = p}K_{p}(z_0,z)\lbrack W_{0,2}(z,z_1)W_{0,2}(\sigma_{p}(z),z_2)\\
    \nonumber &\ \ \ \ \ \ \ \ \ \ \ + W_{0,2}(\sigma_{p}(z),z_1)W_{0,2}(z,z_2)\rbrack\\
    \nonumber\\ 
    W_{1,1}(z_{0}) &= \sum_{p\in R}\mathrm{Res}_{z = p}K_{p}(z_0,z)W_{0,2}(\sigma_{p}(z),z)
\end{align}

$k=3:$
\begin{align}
    W_{0,4}(z_{0},z_1,z_2,z_3) &= \sum_{p\in R}\mathrm{Res}_{z = p}K_{p}(z_0,z)\times\\
    \nonumber     &\lbrack W_{0,3}(z,z_1,z_2)W_{0,2}(\sigma_{p}(z),z_3) + W_{0,3}(\sigma_{p}(z),z_1,z_2)W_{0,2}(z,z_3)\\
    \nonumber     &W_{0,3}(z,z_2,z_3)W_{0,2}(\sigma_{p}(z),z_1) + W_{0,3}(\sigma_{p}(z),z_2,z_3)W_{0,2}(z,z_1)\\
    \nonumber&W_{0,3}(z,z_1,z_3)W_{0,2}(\sigma_{p}(z),z_2) + W_{0,3}(\sigma_{p}(z),z_1,z_3)W_{0,2}(z,z_2))\rbrack
\end{align}
    
\begin{align}
    W_{1,2}(z_{0},z_{1}) = \sum_{p\in R}&\mathrm{Res}_{z = p}K_{p}(z_0,z)\lbrack W_{0,3}(z,\sigma_{p}(z),z_1) \\
    \nonumber &+ W_{1,1}(z)W_{0,2}(\sigma_{p}(z),z_1) + W_{1,1}(\sigma_{p}(z))W_{0,2}(z,z_1)\rbrack   
\end{align}

\vspace{10pt}

The $W_{g,n}$'s satisfy three properties that will be used in future computations.
\begin{enumerate}
    \item \textbf{Symmetry}
    
    For any $\sigma_{ij} \in S_{2}$ (a map that swaps $i$ and $j$), the differentials satisfy:
    \begin{equation}
        W_{g,n}(\sigma_{ij}(\textbf{z})) = W_{g,n}(\textbf{z})
    \end{equation}
    
    \item \textbf{Location of poles}
    
    If a stable differential has a pole then it must be at a ramification point.
    
    \item \textbf{Stable differentials are odd}
    
    The differentials $W_{g,n}$ satisfy
    \begin{equation}
        W_{g,n}(z_{1},...,z_{n}) + W_{g,n}(-z_{1},...,z_{n}) = 0
    \end{equation}
    
    for $(g,n) \neq (0,2)$.  By symmetry, this property extends to all arguments.
    
    The differential $W_{0,2}$ satisfies a different property,
    \begin{equation}
        W_{0,2}(z_{1},z_{2}) + W_{0,2}(-z_{1},z_{2}) = \frac{dx(z_{1})dx(z_{2})}{(x(z_{1}) - x(z_{2}))^{2}}. \label{EqnW02}
    \end{equation}
\end{enumerate}

\subsection{Topological recursion for Hitchin spectral curves}

The topological recursion developed above is a local formulation.  This was adapted to a recursion to a global spectral curves, such as a Hitchin spectral curves, by Dumitrescu-Mulase (c.f.\cite{DumitrescuMulase14, DumitrescuMulase14b, DumitrescuMulase17, DumitrescuMulase18}).  If we want to use a Hitchin spectral curve (or any globally defined spectral curve that satisfies the necessary conditions) as a component for the recursion, we need to replace all components of the recursion with global objects.\\

Let $\tilde{\pi}:\tilde{S}\to X$ be a non-singular cover of a Riemann surface $X$ (not necessarily arising from a Higgs bundle) with simple ramification points.  Denote by $R$, the ramification divisor of $\pi$.  The cover is a Galois covering with Galois group $\mathbb{Z}/2\mathbb{Z} = \langle \widetilde{\sigma} \rangle$, whose fixed point divisor is $R$.  Choose a spin structure on $S$ such that
\begin{equation}
    \dim H^{0}(\tilde{S}, K_{\tilde{S}}^{\frac{1}{2}}) = 1,
\end{equation}
and a symplectic basis $\langle A_{1},...,A_{\tilde{g}},B_{1},...,B_{\tilde{g}} \rangle$ for $H_{1}(\tilde{S},\mathbb{Z})$.

\begin{dfn}
    The \textbf{Eynard-Orantin differentials} $W_{g,n}$ are meromorphic sections of the $n$-th exterior tensor product $K^{\boxtimes n}$ defined as follows:
    \begin{itemize}
        \item $W_{0,1}$ is a meromorphic $1$-form on $\tilde{S}$ prescribed according to the geometric setting.
        \item $W_{0,2}$ is given by
        \begin{equation}
            W_{0,2}(z_{1},z_{2}) = B(z_{1},z_{2}).
        \end{equation}
    \end{itemize}
    For all $g,n\in \N$ and $2g - 2 + n \geq 0$, define $W_{g,n}$ recursively by
    \begin{align}\label{EqnHitchinTRFormula}
        W_{g,n+1}(z_0, \textbf{z}) = \sum_{p\in R}&\mathrm{Res}_{z = p}\frac{\omega^{z-\widetilde{\sigma}(z)}(z_{1})}{\Omega(z)}\Big[ W_{g-1,n+2}(z, \sigma_{p}(z),\textbf{z})\\
        \nonumber&+\sum_{\substack{g_1+g_2=g \\ I\cup J = \textbf{z}}}^{'}W_{g_1, |I|+1}(z,I)W_{g_2, |J|+1}(\sigma_{p}(z), J)\Big]
    \end{align}
    where $\omega$ is the normalized Cauchy kernel (Definition \ref{DfnNormalizedCauchyKernel}), $\Omega = W_{0,1} - \sigma^{*}W_{0,1}$, and the prime signifies summation excluding the cases $(g_1,I)$ or $(g_2,J) = (0,0)$.
\end{dfn}

\begin{remark}
    \emph{As one might expect, this definition is remarkably similar to Definition \ref{DfnOGTR}.  The major difference is that the components going into the definition, such as the spectral curve and the recursion kernel, have been defined in a coordinate-independent global manner.  Nevertheless, it is important to remark that the recursion is inherently a local procedure, as it is built around the data of residues at the ramification points of the spectral curve.  On local coordinate charts, where we can express $S$ as the zero locus of a polynomial, this definition coincides with ordinary Eynard-Orantin topological recursion.}
\end{remark}

\begin{remark}
    \emph{When we have a non-singular Hitchin spectral curve $S$, we let $\tilde{S} = S$, $\tilde{\sigma} = \sigma$, the involution of $T^{*}X$ is defined by fiberwise multiplication by $-1$, and $W_{0,1}=\eta$.}
\end{remark}


\section{Geometry of twisted Hitchin spectral curves}

\subsection{Insights on an $\mathcal{L}$-twisted topological recursion}

The Dumitrescu-Mulase picture outlined at the end \emph{Section 3} fits into the larger scope of $\mathcal{L}$-twisted Higgs bundles, which are adaptable to all Riemann surfaces --- in fact, all complex manifolds --- without having to introduce punctures or special divisors.  Notably, Dumitrescu-Mulase circumvent the meromorphic data when discussing topological recursion by defining the recursion on a non-singular model for the Hitchin spectral curve.  To our knowledge, twisted Higgs bundles have not been explored in the context of topological recursion.  Our interest in them is due to the potential they have for formulating a more invariant version of topological recursion and revealing intricacies that might otherwise not be clear in the already explored cases.\\

To this end, let $(\mathcal{E},\phi)$ be an $\mathcal{L}$-twisted Higgs bundle on a Riemann surface $X$.  The characteristic polynomial of $\phi$ gives rise to a spectral curve $S\subset Tot(\mathcal{L})$ with equation
\begin{equation}
    \eta^{\otimes r} + \sum_{i=1}^{r}p_{i}\eta^{\otimes(r-i)} = 0.
\end{equation}
We can proceed in a na\"ive way by reframing topological recursion for Hitchin spectral curves in the $\mathcal{L}$-twisted setting by replacing the appropriate objects with their $\mathcal{L}$-twisted versions. This $\mathcal{L}$-twisted recursion will take place in the diagram
\begin{center}
\begin{tikzpicture}
    \matrix (m) [matrix of math nodes, row sep=3em, column sep=1em]
    { K_{S} &   & \\
      S &  & Tot(\mathcal{L}) \\
      & X & \\ };
    \path[-stealth]
    (m-1-1) edge node {} (m-2-1)
    (m-2-1) edge   (m-2-3)
            edge node [left] {$\pi|_{S}$} (m-3-2)
    (m-2-3) edge node [right] {$\pi$} (m-3-2);
\end{tikzpicture}
\end{center}
with a Galois involution $\sigma$ of $S$, and ramification divisor $R$ of $\pi$.\\

As the Bergman kernel $B$ depends only on the Riemann surface $S$, we choose $W_{0,2} = B$.  In the $\mathcal{L}$-twisted setting, the recursion kernel at $p\in R$,
\begin{equation}
    K_{p} = \frac{\omega^{z-\sigma(z)}}{(\eta - \sigma^{*}\eta)}
\end{equation}
is a section of $(\pi^{*}\mathcal{L}|_{S})^{*}\otimes K_{S}$.  Using the recursion formula \eqref{EqnHitchinTRFormula} gives rise to \emph{$\mathcal{L}$-twisted Eynard-Orantin differentials},
\begin{equation}
    W_{g,n}^{\mathcal L} \in \Gamma(((\pi^{*}\mathcal{L}|_{S})^{*}\otimes K_{S})^{\boxtimes 2g+n-2}\otimes K_{S}^{n}).
\end{equation}
A notable observation made apparent in the twisted case is the $(\pi^{*}\mathcal{L}|_{S})^{*}\otimes K_{S}$ part of the differentials coming from the recursion kernel.  From this observation, it is clear that the recursion kernel introduces the geometry of $\textrm{Tot}(\mathcal{L})$ into the recursion.  In the original formulation of topological recursion, the recursion kernel \eqref{EqnRecursionKernel} introduces the geometry of $T^{*}\mathbb{P}^{1}$ through the canonical one-form $ydx$ in the denominator, although it does not appear as obviously when looking at where the $W_{g,n}$'s live, as the embedding of $S$ into $K_{X}$ made $K_{S} = \pi^{*}K_{X}|_{S}$, and so the two terms cancel out.\\

With this na\"ive approach to framing an $\mathcal{L}$-twisted topological recursion, we have already revealed something about the $W_{g,n}$ that was not present in the regular case.  A potentially interesting result of this generalization is a relationship between the $\mathcal{L}$-twisted Eynard-Ortantin differentials and the hyperk\"ahler structure on the moduli space of Higgs bundles.  While the moduli space of ordinary Higgs bundles in genus $g \geq 2$ possesses a canonical holomorphic symplectic structure (which derives from the hyperk\"ahler structure), moduli spaces of Higgs bundles with twists have Poisson structures that come in a family, no element of which is necessarily canonical.  The component of $W_{g,n}^{\mathcal{L}}$ in $\pi^{*}\mathcal{L}^{-1}|_{S}\otimes K_{S}$ defines a Poisson structure on a moduli space of $\pi^{*}\mathcal{L}^{-1}|_{S}$-twisted Higgs bundles defined on the spectral curve.  This observation begs two important ideas: in the ordinary topological recursion, the holomorphic symplectic form may be encoded in the Eynard-Orantin differentials (a partial result to this effect may already be present in the work of \cite{BaragliaHuang17}, who extract the so-called Donagi-Markman cubic from the differentials); and that Higgs bundles defined on the spectral curve itself may be relevant to topological recursion.\\

These observations serve as a reason to explore the na\"ive $\mathcal{L}$-twisted topological recursion.  To validate the framework that is being developed, we could appeal to one of the three main areas to which topological recursion has been applied: enumerative geometry, quantum curves, or the geometry of the Hitchin moduli space.  Any of these areas would prove meaningful and interesting for this framework, but as the main interest of this section lies in the relationship between Higgs bundles and topological recursion, we will choose the latter and study the relationship between $\mathcal{L}$-twisted Hitchin spectral curves and $\mathcal{L}$-twisted topological recursion.

\subsection{$\mathcal{L}=K$ Hitchin moduli space}

With our goal being to validate a framework for twisted topological recursion, we want to understand how (ordinary) topological recursion relates to the geometry of the (ordinary) Hitchin moduli space.  In this section, we return to the $\mathcal{L} = K$ setting and briefly go through the main results of \cite{BaragliaHuang17}, which we will seek to replicate in some capacity in the $\mathcal{L}$-twisted setting in subsequent sections.\\

In the $\mathcal{L}=K$ case, the moduli space of Higgs bundles $\mathcal{M}_{X}$ admits a complete hyperk\"ahler metric \cite{Hitchin87}.  This metric can be studied by studying a second related hyperk\"ahler metric, called the \emph{semi-flat metric}, defined over the regular locus of the Hitchin fibration $\mathcal{H}:\mathcal{M}_{X}^{reg}\to\mathcal{B}^{reg}$ (i.e. the locus containing non-singular fibres of $\mathcal{H}$).  A theorem of Hitchin \cite{Hitchin99} says that under a set of mild assumptions, $\mathcal{B}^{reg}$ admits a special K\"ahler structure. The special K\"ahler metric on $\mathcal{B}^{reg}$ can be combined with a metric along the fibers to produce the semi-flat metric \cite{CecottiFerraraGiradello89, Freed99, Hitchin99}.\\

On $\mathcal{B}^{reg}$, we can define a local \emph{conjugate coordinate system} $\{z_{1},\dots,z_{g_{S}} \}$ and $\{w_{1},\dots, w_{g_{S}} \}$, where $S$ is the spectral curve associated to a point in $\mathcal{B}^{reg}$, by
\begin{align}
    z_{i} &= \int_{A_{i}}\theta,\\
    w_{j} &= \int_{B_{j}}\theta,
\end{align}
where $\theta$ is the canonical one-form on $S$, and $\langle A_{1},...,A_{g},B_{1},...,B_{g} \rangle$ for $H_{1}(X,\mathbb{Z})$ is a symplectic basis.  With respect to the \emph{special coordinate system} $\{z_{i}\}$, the K\"ahler form $\omega$ can be written as
\begin{equation}\label{EqnSpecialKahlerForm}
    \omega = \frac{i}{2}Im(\tau_{ij})dz^{i}\wedge d\bar{z}^{j},
\end{equation}
where $\tau_{ij}$ is the period matrix of the spectral curve.  This means that the K\"ahler metric on $\mathcal{B}^{reg}$, and therefore the semi-flat metric on $\mathcal{M}_{X}$, can be written in terms of the period matrices $\tau_{ij}$ of spectral curves through \eqref{EqnSpecialKahlerForm}.\\

We can apply the topological recursion for Hitchin spectral curves defined in \emph{Section 3} to spectral curves associated to $\mathcal{B}^{reg}$. In this setting, the variational formula for the Eynard-Orantin invariants \cite{EynardOrantin07b} provides a relationship between derivatives of the period matrix $\tau_{ij}$ about a point $b\in\mathcal{B}^{reg}$ and the $g=0$ invariants $W_{0.n}$ of the spectral $S_{b}$ associated to the point $b$.

\begin{thm}[Baraglia-Huang]\label{ThmBH}
    \begin{equation}        \partial_{i_{1}}\partial_{i_{2}}\dots\partial_{i_{m-2}}\tau_{i_{m-1}i_{m}} = -\left( \frac{i}{2\pi} \right)^{m-1} \int_{p_{i_{1}}\in b_{i_{1}}}\dots\int_{p_{m}\in b_{i_{m}}} W_{0,m}(p_{1},...,p_{m})
    \end{equation}
\end{thm}

This formula shows that the $g=0$ invariants for the spectral curve $S_{b}$ compute the Taylor series expansion of the period matrix about $b\in\mathcal{B}^{reg}$, and by extension of what was said above, the semi-flat metric on $\mathcal{M}_{X}^{reg}$.  Notably, Theorem \ref{ThmBH} states that the data of single spectral curve is enough to understand the geometry of nearby spectral curves.\\

In the $\mathcal{L} \neq K$ setting, the moduli space $\mathcal{M}^{\mathcal{L}}_{X}$ is no longer hyperk\"ahler, so we cannot expect to make claims relating the $g=0$ invariants to a hyperk\"ahler metric on the moduli space, however, we can (and will) seek to find an analogy for Theorem \ref{ThmBH}.\\

There are two concerns that impede our ability to immediately replicate the result.  Firstly, in our twisted setting, the dimension of $\mathcal{M}_{X}^{\mathcal{L}}$ is not twice the dimension of the Hitchin base $\mathcal{B}$.  This means that the base and fibres of the moduli space no longer have the same dimension. In fact, the dimension of the base is larger than that of the fibres, so we can no longer view $\mathcal{B}^{reg}$ as the deformation space of spectral curves.  Secondly, in the $K$-case, we have a canonical coordinate system on $T^{*}X$ given by the natural symplectic structure, which allows for many computations to be simplified, and relates the tautological section to the geometry of the total space.  Without imposing a symplectic structure on $\text{Tot}(\mathcal{L})$, and thus limiting our breadth of line bundles, we need to find a sufficiently natural choice of coordinate system for $\text{Tot}(\mathcal{L})$.  We seek to address these concerns by studying more in-depth the $\mathcal{L}$-twisted moduli space.


\subsection{Deformation theory for the $\mathcal{L}$-twisted moduli space}\label{SectionDeformationTheory}

With our attention focused on $\mathcal{L}$-twisted Higgs bundles, we aim to describe the local structure of the moduli space $\mathcal{M}_{X}^{\mathcal{L}}$.  In this section, we will intiate a thorough examination of the deformation theory $\mathcal{M}_{X}^{\mathcal{L}}$ by studying its hypercohomology (\emph{cf.} \cite[pp. 438-447]{GriffithsHarris94}).\\

On a Riemann surface $X$, an $\mathcal{L}$-twisted Higgs field satisfies
\begin{equation*}
    \phi\wedge\phi = 0 \in H^{0}(X, \text{End}(\mathcal{E})\otimes \wedge^{2}\mathcal{L}).
\end{equation*}
There is a natural sequence of sheaves associated to a Higgs bundle $(\mathcal{E}, \phi)$, given by
\begin{equation}
    \text{End}(\mathcal{E}) \xrightarrow{-\wedge\phi} \text{End}(\mathcal{E})\otimes \mathcal{L} \xrightarrow{-\wedge\phi} \text{End}(\mathcal{E})\otimes\wedge^{2}\mathcal{L} \xrightarrow{-\wedge\phi} \dots,
\end{equation}
where $-\wedge\phi$ acts by the adjoint representation on the $\text{End}(\mathcal E)$ part.  Because $\phi\wedge\phi = 0$, we have that $-\wedge\phi$ is a differential, and so we can use this resulting complex to define cohomologies on $\text{End}(\mathcal{E})\otimes^{i} \mathcal{L}$.  This hypercohomology intertwines to cohomology theories, namely the cohomology of sheaves arising from the \v Cech complex and the cohomology arising from the Higgs field.\\

To be specific, we consider the hypercohomologies associated to a stable $\mathcal{L}$-twisted Higgs bundle $(\mathcal{E}, \phi)$ on a Riemann surface $X$ given respectively by the two double complexes
\begin{align*}
    D &= (\delta, \wedge\phi),\\
    D' &= (\wedge\phi, \delta),
\end{align*}
where $\wedge\phi$ is the differential coming from the Higgs field, and $\delta$ is the \v Cech differential.\\

We start by computing the hypercohomology with the complex $D$.  The hypercohomology fits into the short exact sequence
\begin{equation}\label{EqnSESD}
    D:\ 0 \rightarrow E^{1,0} \rightarrow \mathbb{H}^{1} \rightarrow E^{0,1} \rightarrow 0,
\end{equation}
where
\begin{align*}
    E^{1,0} &= \frac{\text{ker} H^{0}(\text{End}\mathcal{E}\otimes \mathcal{L})\xrightarrow{\wedge\phi} H^{0}(\text{End}\mathcal{E}\otimes \mathcal{L}^{\otimes 2})}{\text{im} H^{0}(\text{End}\mathcal{E})\xrightarrow{\wedge\phi} H^{0}(\text{End}\mathcal{E}\otimes \mathcal{L})},\\
    E^{0,1} &= \text{ker} H^{1}(\text{End}\mathcal{E}) \xrightarrow{\wedge\phi} H^{1}(\text{End}\mathcal{E}\otimes \mathcal{L}).
\end{align*}
Here, $\mathbb{H}^{1} = T_{(\mathcal{E},\phi)}\mathcal{M}^{\mathcal{L}}_{X}(r,d)$ is the tangent space to the moduli space of rank $r$, degree $d$, $\mathcal{L}$-twisted Higgs bundles.  This is a short exact sequence around $\mathbb{H}^{1}$ because of vanishing due to stability and dimensionality ($X$ is a curve and $rk\mathcal{L}=1$).\\

To determine $\mathbb{H}^{1}$, we first need to better understand $E^{1,0}$ and $E^{0,1}$.  We start by looking at $E^{0,1}$.  Using Serre Duality, we  compute the dimension of $H^{1}(\text{End}\mathcal{E}\otimes \mathcal{L})$ as
\begin{equation}\label{EqnDegreeH1is0}
    h^{1}(\text{End}\mathcal{E}\otimes \mathcal{L}) = h^{0}(K\otimes \text{End}\mathcal{E}\otimes \mathcal{L}^{*}) = 0,
\end{equation}
which vanishes because $\deg\mathcal{L} > \deg K$.  This means $E^{0,1} = H^{1}(\text{End}\mathcal{E}) = \mathcal{U}_{X}(r,d)$, which is the moduli space of stable bundles on $X$.  Its dimension is then
\begin{equation}\label{EqnDimE01}
    \dim E^{0,1} = \dim\mathcal{U}_{X}(r,d) = r^{2}(g-1)+1.
\end{equation}

Turning our attention to $E^{1,0}$, observe that
\begin{equation*}
    E^{1,0} = \frac{H^{0}(\text{End}\mathcal{E}\otimes \mathcal{L})}{\text{im}(\wedge\phi)}.
\end{equation*}
The map $\wedge\phi$ has kernel generated by $1$ when acting on $H^{0}(\text{End}\mathcal{E})$, as stability implies simplicity.  Thus we have
\begin{align*}
    \dim E^{1,0} &= h^{0}(\text{End}\mathcal{E}\otimes \mathcal{L}) - h^{0}(\text{End}\mathcal{E}) + 1.
\end{align*}
To handle the first term, $h^{0}(\text{End}\mathcal{E}\otimes \mathcal{L})$, we apply Riemann-Roch and Serre Duality, recalling from \eqref{EqnDegreeH1is0} that $h^{1}(\text{End}\mathcal{E}\otimes \mathcal{L}) = 0$,
\begin{align*}
    h^{0}(\text{End}\mathcal{E}\otimes \mathcal{L}) &= \deg(\text{End}\mathcal{E}\otimes \mathcal{L}) + \text{rk}(\text{End}\mathcal{E}\otimes \mathcal{L})(1-g) - h^{1}(\text{End}\mathcal{E}\otimes \mathcal{L})\\
    &= r^{2}\deg L + r^{2}(1-g).
\end{align*}
For the second term, $h^{0}(\text{End}\mathcal{E})$, if we assume that $\mathcal{E}$ is stable as a vector bundle (which is generically true), then $h^{0}(\text{End}\mathcal{E}) = 1$.
We compute the $\dim E^{1,0}$ as
\begin{align}
    \nonumber\dim E^{1,0} &= h^{0}(\text{End}\mathcal{E}\otimes L) - h^{0}(\text{End}\mathcal{E}) + 1\\
    \nonumber&= r^{2}\deg\mathcal{L} + r^{2}(1-g) - 1 + 1\\
    &= r^{2}\deg\mathcal{L} + r^{2}(1-g).\label{EqnDimE10}
\end{align}

Combining \eqref{EqnDimE01} and \eqref{EqnDimE10}, the dimension of $\mathbb{H}^{1}$ is
\begin{equation*}
    \dim\mathbb{H}^{1} = r^{2}\deg\mathcal{L} + r^{2}(1-g) + r^{2}(g-1)+1 = r^{2}\deg\mathcal{L} + 1.
\end{equation*}
Because $\mathbb{H}^{1} = T_{(\mathcal{E},\phi)}\mathcal{M}^{\mathcal{L}}_{X}(r,d)$, we conclude that the dimension of $\mathcal{M}^{\mathcal{L}}_{X}(r,d)$ is
\begin{equation}
    \dim\mathcal{M}^{\mathcal{L}}_{X}(r,d) = r^{2}\deg\mathcal{L} + 1,
\end{equation}
which agrees with the calculations from \cite{Nitsure91}.\\

Denote by $\mathbb{L}$, the bundle over $\mathcal{U}_{x}(r,d)$ whose fibre at $\mathcal{E}$ is $H^{0}(\text{End}\mathcal{E}\otimes \mathcal{L})$.  By the fact that $(\mathcal{E}, \phi)$ is stable for all $\phi\in H^{0}(\text{End}\mathcal{E}\otimes \mathcal{L})$, and $\text{Aut}(\mathcal{E})$ acts trivially on $H^{0}(\text{End}\mathcal{E}\otimes \mathcal{L})$, we have an injection $\text{Tot}\left(\mathbb{L}\right) \hookrightarrow \mathcal{M}^{\mathcal{L}}_{X}(r,d)$ that is open and dense (as their dimensions equal).  In other words, $\mathcal{M}^{\mathcal{L}}_{X}(r,d)$ is the completion of $\mathbb{L}$ under stability, and
\begin{align}
    \nonumber T_{(\mathcal{E},\phi)}\mathcal{M}^{\mathcal{L}}_{X}(r,d) &\cong T_{(\mathcal{E},\phi)}\mathbb{L}\\
    &\cong \underbrace{H^{0}(\text{End}\mathcal{E}\otimes L)}_{\mathbb{L}_{\mathcal{E}}} \times \underbrace{H^{1}(\text{End}E)}_{T_{\mathcal{E}}\mathcal{U}_{X}(r,d)},\label{EqnHSplitting1}
\end{align}
whenever $\mathcal{E}$ is a stable bundle.  This means that $\mathbb{L}$ is the $\mathcal{L}$-twisted analogue of $T^{*}\mathcal{U}_{X}(r,d)$.

\begin{remark}
\emph{The dimension of $\mathcal{M}^{\mathcal{L}}_{X}(r,d)$ is actually independent of the stability of $\mathcal{E}$\footnote{The importance of stability in the previous calculations was to reveal the relationship to $\mathbb{L}$.}.  If we do not assume that $\mathcal{E}$ is stable, then}
\begin{align*}
    \dim\mathbb{H}^{1} &= h^{0}(\textrm{End}\mathcal{E}\otimes \mathcal{L}) - h^{0}(\text{End}\mathcal{E}) + 1 + h^{1}(\text{End}\mathcal{E})\\
    &= r^{2}\deg\mathcal{L} + r^{2}(1-g) + 1 - (h^{0}(\text{End}\mathcal{E}) - h^{1}(\text{End}\mathcal{E}))\\
    &= r^{2}\deg\mathcal{L} + 1 + r^{2}(1-g) - r^{2}(1-g)\\
    &= r^{2}\deg\mathcal{L} + 1.
\end{align*}
\end{remark}

We now consider the hypercohomology of $D'$.  We have the same type of short exact sequence as \eqref{EqnSESD}, with $E^{1,0}$ and $E^{0,1}$ now given by
\begin{align*}
    E^{1,0} &= H^{1}\left(\ker(\textrm{End}\mathcal{E} \xrightarrow{\wedge\phi} \text{End}\mathcal{E}\otimes \mathcal{L})\right),\\
    E^{0,1} &= H^{0}\left(\frac{\text{End}\mathcal{E}\otimes \mathcal{L}}{\text{im}(\wedge\phi)}\right).
\end{align*}

Suppose that $(\mathcal{E}, \phi)$ is not only stable but also regular, which means that $ker(\wedge\phi)$ is minimally generated, i.e. $ker(\wedge\phi)$ is a rank $r$ subsheaf of $End\mathcal{E}$ --- this is actually a generic property.
Consider the compositions $\phi^{\otimes i}: \mathcal{E} \to \mathcal{E}\otimes \mathcal{L}^{\otimes i}$.  Note that
\begin{equation*}
    \phi^{i}\in Hom(\mathcal{E}, \mathcal{E}\otimes \mathcal{L}^{i}) = Hom(\mathcal{L}^{-i}, \mathcal{E}^{*}\otimes \mathcal{E})
\end{equation*}
and so
\begin{equation*}
    \phi^{i}: \mathcal{L}^{-1} \to End(\mathcal{E}).
\end{equation*}
Notice that
\begin{equation*}
    \phi^{i} \wedge \phi = [\phi^{i}, \phi] = 0,
\end{equation*}
so we have $\phi^{0}, \phi^{1}, \dots, \phi^{r-1}\in \ker(\wedge\phi)$, and in fact, they generate the sheaf $\ker(\wedge\phi)$, meaning that
\begin{equation*}
    \ker(\wedge\phi) = \mathcal{O}\oplus \mathcal{L}^{-1}\oplus\cdots\oplus \mathcal{L}^{-(r-1)}.
\end{equation*}

The sheaf $\frac{\text{End}\mathcal{E}\otimes \mathcal{L}}{\text{im}(\wedge\phi)}$ is the cokernel of $\wedge\phi$.  Starting from the exact sequence
\begin{equation*}
    0 \rightarrow \ker\wedge\phi \xrightarrow{\phi^{i}} \text{End}\mathcal{E} \xrightarrow{\wedge\phi} \text{End}\mathcal{E}\otimes \mathcal{L} \rightarrow \text{coker}\wedge\phi \rightarrow 0,
\end{equation*}
we dualize the sequence to obtain
\begin{equation*}
    0 \rightarrow (\text{coker}\wedge\phi)^{*} \rightarrow \text{End}\mathcal{E}\otimes \mathcal{L}^{-1} \xrightarrow{\phi^{*}\wedge} \text{End}\mathcal{E} \rightarrow (\ker\wedge\phi)^{*} \rightarrow 0,
\end{equation*}
and tensor by $\mathcal{L}$ to return to the original sequence
\begin{equation}
   0 \rightarrow (\text{coker}\wedge\phi)^{*}\otimes \mathcal{L} \rightarrow \text{End}\mathcal{E} \rightarrow \text{End}\mathcal{E}\otimes \mathcal{L} \rightarrow (\ker\wedge\phi)^{*}\otimes \mathcal{L} \rightarrow 0
\end{equation}
and so, we see that
\begin{equation*}
    \text{coker}\wedge\phi = (\ker\wedge\phi)^{*} =  \mathcal{L}\oplus \mathcal{L}^{2}\oplus\cdots\oplus \mathcal{L}^{r}.
\end{equation*}

This provides a different splitting of $\mathbb{H}^{1}$ as
\begin{equation}\label{EqnHSplitting2}
    \mathbb{H}^{1} = H^{1}\left(\bigoplus^{r-1}_{i=0} \mathcal{L}^{-i}\right)\times H^{0}\left(\bigoplus^{r}_{i=1} \mathcal{L}^{i}\right).
\end{equation}
This splitting is induced by the derivative of the Hitchin map,
\begin{equation*}
    \mathcal{H}: \mathcal{M}^{\mathcal{L}}_{X}(r,d) \to \mathcal{B},
\end{equation*}
where $\mathcal{B} = H^{0}\left(\bigoplus^{r}_{i=1} \mathcal{L}^{i}\right)$.  Recall from Section \ref{SubsectionSpectralCorrespondence}, that $\mathcal{M}_{X}^{\mathcal{L}}$ is fibred by Jacobians of spectral curves.  In other words, \eqref{EqnHSplitting2} tells us that
\begin{align*}
    T_{(\mathcal{E},\phi)}\mathcal{M}_{X}^{\mathcal{L}}(r,d) &\cong T_{\mathcal{Q}}Jac(S)\times T_{\mathcal{H}(\mathcal{E},\phi)}\mathcal{B}\\
    &= H^{1}\left(\bigoplus^{r-1}_{i=0} \mathcal{L}^{-i}\right)\times \mathcal{B},
\end{align*}
where $S$ is the spectral curve of $(\mathcal{E},\phi)$, and $\mathcal{Q}$ is the spectral line bundle.  In particular, the tangent space to a Hitchin fiber over the regular locus (call this $\mathcal{B}^{reg}$) is $\cong H^{1}\left(\bigoplus^{r-1}_{i=0} \mathcal{L}^{-i}\right)$, and the genus of the spectral curve is given by
\begin{equation*}
    g_{S} = \dim Jac(S) = \sum_{i=0}^{r-1} h^{1}(X, \mathcal{L}^{i}).
\end{equation*}


Serre duality on $X$ pairs $T_{M}Jac(S) \cong H^{1}\left(\bigoplus^{r-1}_{i=0} \mathcal{L}^{-i}\right)$ with $H^{0}\left(\bigoplus^{r-1}_{i=0} \mathcal{L}^{i}\otimes K\right)$.  We can map this space to a subvariety $\widetilde{\mathcal{B}}\subset\mathcal{B}$ by a choice of section $s\in H^{0}(X, K^{*}\otimes \mathcal{L})\backslash \{0\}$, which acts by multiplication.  Recall that in the $\mathcal{L}=K$ setting, the Hitchin base has the same dimension as fibres.  The vector space $H^{0}\left(\bigoplus^{r-1}_{i=0} \mathcal{L}^{i}\otimes K\right)$ is fulfilling this role in the $\mathcal{L}$-twisted picture.

\begin{dfn}
We call $\mathcal{B}_{eff} \coloneqq H^{0}\left(\bigoplus^{r-1}_{i=0} \mathcal{L}^{i}\otimes K\right)$ the \textbf{effective Hitchin base}.
\end{dfn}

By duality,
\begin{equation*}
    \dim \mathcal{B}_{eff} = \dim Jac(S) = g_{s},
\end{equation*}
which means $h^{-1}(\mathcal{B}_{eff})$ is a moduli space of $\mathcal{L}$-twisted Higgs bundles in which the fibre and base are equidimensional.\\

In particular, when $\mathcal{L}=K$, $s\in H^{0}(K^{*}\otimes K)\backslash \{0\}$ is a nonzero multiple of the identity, and $\mathcal{B}_{eff} = \mathcal{B}$.

\begin{ex}
Consider the situation where $X = \mathbb{P}^{1},\ r = 2,\ d = -1,\ \mathcal{L}=\mathcal{O}(2)$.  The moduli space has dimension
\begin{equation*}
    \dim\mathcal{M}_{\mathbb{P}^{1}}^{\mathcal{O}(2)}(2, -1) = 2^{2}deg(\mathcal{O}(2)) + 1 = 2^{2}(2) + 1 = 9.
\end{equation*}
The Hitchin base is
\begin{equation*}
    \mathcal{B} = H^{0}(\mathcal{O}(2))\oplus H^{0}(\mathcal{O}(4)) \cong \mathbb{C}^{8},
\end{equation*}
and so the moduli space is fibred by $1$-dimensional (elliptic) fibres.

\noindent The effective Hitchin base, on the other hand, is
\begin{equation*}
    \mathcal{B}_{eff} = H^{0}(\mathcal{O}(-2))\oplus H^{0}(\mathcal{O}(2)\otimes\mathcal{O}(-2)) \cong \mathbb{C},
\end{equation*}
which has the same dimension as the fibre, and which is also that of the moduli space of elliptic curves.\\

To see $\mathcal{B}_{eff}$ in $\mathcal{B}$, we need to choose an $s\in H^{0}(\mathcal{O}(-2)^{*}\otimes\mathcal{O}(2))\backslash 0 = H^{0}(\mathcal{O}(4))\backslash 0$.  Given a Higgs bundle $(E,\phi)$, there is a canonical choice; in this case: $s = \det\phi$.  With this choice of $s$, $\mathcal{B}_{eff}$ is embedded as the line
\begin{equation*}
    \widetilde{\mathcal{B}} = \{ (0, c\det\phi)\ |\ c\in\mathbb{C} \} \subset \mathcal{B}.
\end{equation*}

\begin{figure}[htp] 
    \centering
    \includegraphics[width=9cm]{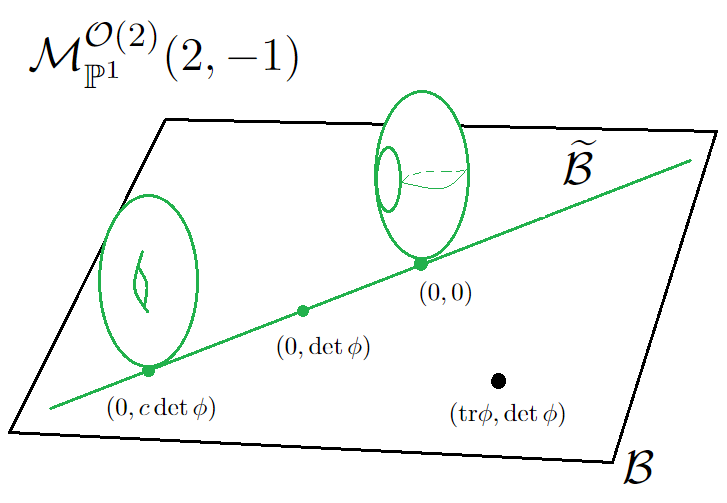}
    \caption{$\mathcal{B}_{eff}$ mapped into $\mathcal{B}$ by $s=\det\phi$.}
    \label{FigTwistedModuliExample}
\end{figure}

\end{ex}

The effective Hitchin base $\mathcal{B}_{eff}$ serves as a moduli space of deformations spectral curves and can be related to the period data as follows.

Let $\pi: Tot(\mathcal{L})\to X$ be the natural projection.  The derivative of $\pi|_{S}$ is
\begin{equation*}
    d\pi : TS \to \pi^{*}TX \in Hom(\pi^{*}K, K_{S}),
\end{equation*}
where $K = K_{X}$.  Note that $\pi^{*}s: \pi^{*}K\to \pi^{*}\mathcal{L}$, so if $\eta$ is the tautological section in $H^{0}(S, \pi^{*}\mathcal{L})$, then we have that $(\pi^{*}s)^{-1}\eta$ is a (meromorphic) section of $\pi^{*}K$.  We can define a meromorphic $1$-form on $S$ by
\begin{equation}
    \Theta = d\pi((\pi^{*}s)^{-1}\eta)\in H^{0}(S, K_{S}).
\end{equation}

\begin{dfn}
    The meromorphic $1$-form $\Theta$ is called the \textbf{twisted canonical} $1$-form.
\end{dfn}

If $\{ a_{1},b_{1},\dots, a_{g_{S}}, b_{g_{S}} \}$ is a basis for $H_{1}(S,\Z)$, then $\{ Re\int_{a_{i}(b_{i})}\Theta, Im\int_{a_{i}(b_{i})}\Theta \}$ form local (singular) coordinates on $B_{eff}$.  We may choose a basis $\omega_{1},\dots,\omega_{g_{s}}$ of $H^{0}(S,K_{S})$ so that $\int_{a_{i}}\omega_{j} = \delta_{ij}$, and $\tau = \big[ \int_{b_{i}}\omega_{j} \big]$ is the $g_{S}\times g_{S}$ period matrix.  Define $\lambda_{i} = Re\int_{a_{i}}\Theta + iIm\int_{a_{i}}\Theta$.

\begin{dfn}
    We refer to the collection of derivatives $\frac{\partial\tau_{jk}}{\partial\lambda_{i}}$ as the \textbf{twisted Donagi-Markman cubic}.
\end{dfn}

The spectral curve $S$ itself is a divisor in $Tot(\mathcal{L})$, and the normal bundle of a divisor is given by the line bundle associated to the divisor, which in this case is $\pi^{*}\mathcal{L}^{r}$ (since $S$ is the zero locus of a degree-$r$ polynomial in $\eta\in H^{0}(S,\pi^{*}\mathcal{L})$).  By adjunction,
\begin{equation*}
    K_{S} = K_{Tot(\mathcal{L})|_{S}}\otimes N_{S} = K_{Tot(\mathcal{L})}|_{S}\otimes \pi^{*}\mathcal{L}^{r}.
\end{equation*}
We remark that in the $\mathcal{L}=K$ setting, $K_{Tot(K)}\cong\mathcal{O}_{Tot(K)} \Rightarrow K_{S} = \pi^{*}K^{r}$.\\

Note (cf. \cite{BEAUVILLEA1989SCAT}) that
\begin{equation*}
    \pi_{*}\mathcal{O}_{S} \cong \mathcal{O}\oplus \mathcal{L}^{-1} \oplus \cdots \oplus \mathcal{L}^{-(r-1)}.
\end{equation*}
Tensoring with $\mathcal{L}^{r}$, we have
\begin{equation*}
    \pi_{*}\mathcal{O}_{S}\otimes \mathcal{L}^{r} \cong \mathcal{L} \oplus \mathcal{L}^{2} \oplus \cdots \oplus \mathcal{L}^{r}.
\end{equation*}
Looking at cohomology
\begin{align*}
    H^{0}(X, \pi_{*}\mathcal{O}_{S}\otimes \mathcal{L}^{r}) &\cong \bigoplus_{i=1}^{r} H^{0}(X, \mathcal{L}^{i})\\
    H^{0}(S, \mathcal{O}_{S}\otimes \pi^{*}\mathcal{L}^{r}) &\cong \mathcal{B}\\
    H^{0}(S, \pi^{*}\mathcal{L}^{r}) \cong \mathcal{B},
\end{align*}
by properties of the direct image functor.\\

In other words, deformations of $(\mathcal{E}, \phi)$ in the normal direction to $S$ correspond to deformations along the full Hitchin base, while first-order deformations of $S$ correspond to deformations along the effective Hitchin base.


\subsection{Geometry of $b$-manifolds}

We take a brief digression to discuss the language of $b$-manifolds.  This will prove to be a useful tool when working with Higgs bundles in the twisted setting.  These ideas were pioneered in \cite{melrose1993atiyah} in the context of differential operators on manifolds with boundary.  The category of $b$-manifolds was further developed in \cite{guillemin2014symplectic}, \cite{M_rcu__2014} in a more general setting in order to study so-called log-symplectic structures.  Our interest in $b$-geometry is the basic language of $b$-manifolds, and their $b$-tangent and $b$-cotangent bundles, so we will focus on the concepts of $b$-geometry most relevant to our needs, following \cite{guillemin2014symplectic}.

\begin{dfn}
    A \textbf{$b$-manifold} is a pair $(M,Z)$, where $M$ is an oriented manifold and $Z$ is an oriented co-dimension one submanifold of $M$.  A \textbf{$b$-map} is a smooth map $f:(M_{1},Z_{1}) \to (M_{2},Z_{2})$ such that $f^{-1}(Z_{2}) = Z_{1}$, and $f$ is transverse to $Z_{2}$, i.e.
    \begin{equation*}
        T_{f(p)M_{2}} = Im(d_{p}f) + T_{f(p)}Z_{2},
    \end{equation*}
    for all $p\in Z_{1}$.
\end{dfn}

\begin{remark}
\emph{When the setting is clear, we will write $M$ in place of $(M,Z)$ for convenience.}
\end{remark}

\begin{dfn}
    Let $(M,Z)$ be a $b$-manifold.  A $b$-vector field on $M$ is a vector that is tangent to $Z$ for all $p\in Z$.  Denote the set of $b$-vector fields by $^{b}\mathfrak{X}(M)$.
\end{dfn}

An ``ordinary'' vector field $X\in\mathfrak{X}(M)$ is a $b$-vector field on $(M,Z)$ \emph{iff} for all $p\in Z$, there is a neighbourhood $(U, x_{1},\dots,x_{n})$ where $Z\cap U$ is defined by $x_{1} = 0$ and
\begin{equation*}
    X|_{U} = f_{1}x_{1}\frac{\partial}{\partial x_{1}} + f_{2}\frac{\partial}{\partial x_{2}} + \dots + f_{n}\frac{\partial}{\partial x_{n}}
\end{equation*}

for a unique collection of smooth functions $f_{1},\dots, f_{n}\in C^{\infty}(U)$.  This makes $^{b}\mathfrak{X}(M)$ a locally free $C^{\infty}(M)$-module with local bases given by

\begin{align*}
    &\Big\{ \frac{\partial}{\partial x_{1}}, \dots, \frac{\partial}{\partial x_{n}} \Big\} \ \ \ \ \ \ \ \ \ \ \ \ \ \ \ \ \ \ \ \text{away from }Z,\\
    &\Big\{ x_{1}\frac{\partial}{\partial x_{1}}, \frac{\partial}{\partial x_{2}}, \dots, \frac{\partial}{\partial x_{n}} \Big\} \ \ \ \ \ \ \ \ \ \text{near to } Z.
\end{align*}

An application of the Serre-Swan theorem (c.f. \cite{taylor2002several} Proposition 7.6.5) tells us that the set of $b$-vectors are sections of a vector bundle on $M$.

\begin{dfn}
    Let $(M,Z)$ be a $b$-manifold.  The \textbf{$b$-tangent bundle} of $M$, denoted $^{b}TM$, is the vector bundle whose sections are $^{b}\mathfrak{X}(M)$.
\end{dfn}

Starting with a $b$-vector field $v$ on $M$, if we take the restriction of $v$ to $Z$, we get a vector field $v|_{Z}$ which is tangent to $Z$ for all $p \in Z$. Hence, $v|_{Z}$ defines a tangent vector field on $Z$.  Using this restriction, we have a morphism of $C^{\infty}(Z)$-modules $\Gamma(^{b}TM|_{Z}) \to \Gamma(TZ)$, which is induced by a vector bundle isomorphism
\begin{equation*}
    ^{b}TM|_{Z} \to TZ.
\end{equation*}

The kernel of this isomorphism is a line bundle with a canonical non-trivial section $w$, call the \textbf{normal $b$-vector field} of $M$.  In local coordinates, the vector field $w$ can be written in a coordinate independent way as
\begin{equation*}
    w = x_{1}\frac{\partial}{\partial x_{1}},
\end{equation*}
where $\{ x_{1} = 0\}$ locally defines $Z$.\\

At points $p \in M\backslash Z$, the $b$-tangent space at $p$ coincides with the usual tangent space, and at points $p \in Z$, there is a surjective map
\begin{equation*}
    ^{b}T_{p}M \to T_{p}Z
\end{equation*}
whose kernel is spanned by $w_{p}$, the normal $b$-vector field at $p$.  So we can write

\[ ^{b}T_{p}M = \begin{cases} 
      T_{p}M & p \in M\backslash Z \\
      T_{p}Z\oplus span\{ x_{1}\frac{\partial}{\partial x_{1}}|_{p} \} & p \in Z \\
   \end{cases}.
\]

\begin{remark}
\emph{As an ordinary vector field, $x_{1}\frac{\partial}{\partial x_{1}}$ vanishes along $Z$; however, it is non-vanishing when viewed as a $b$-vector field.  Around $Z$, we can think of $x_{1}\frac{\partial}{\partial x_{1}}$ as a formal object from the view point of $b$-geometry.}
\end{remark}

\begin{dfn}
    Let $(M,Z)$ be a $b$-manifold.  The \textbf{$b$-cotangent bundle} is the vector bundle $^{b}T^{*}M$ dual to $^{b}TM$.
\end{dfn}

At points $p \in M\backslash Z$, we have that $^{b}T^{*}_{p}M = (^{b}T_{p}M)^{*} = (T_{p}M)^{*} = T^{*}_{p}M$ coincides with the usual cotangent space.  At points $p \in Z$, the dual of the map for $b$-tangent spaces above gives us an embedding

\begin{equation*}
    T^{*}_{p} \to ^{b}T^{*}_{p}M
\end{equation*}

whose image is $\{ l\in (^{b}T_{p}M)^{*} | l(w_{p}) = 0 \}$.\\

Let $\{ x_{1},\dots,x_{n} \}$ be coordinates around $p$ such that $Z$ is locally defined by $\{ x_{1} = 0 \}$.  Consider the one-form $\frac{dx_{1}}{x_{1}}$.  At points away from $Z$, $\frac{dx_{1}}{x_{1}}$ is a well-defined one-form.  The pairing of $\mu$ with any $b$-vector field extends smoothly over $Z$ because
\begin{equation*}
    \langle f_{1}x_{1}\frac{\partial}{\partial x_{1}} + f_{2}\frac{\partial}{\partial x_{2}} + \dots + f_{n}\frac{\partial}{\partial x_{n}}, \frac{dx_{1}}{x_{1}} \rangle = f_{1}.
\end{equation*}

This means $\frac{dx_{1}}{x_{1}}$ can be extended over $Z$ as a section of $^{b}T^{*}M$.  We will denote this section as $\frac{dx_{1}}{x_{1}}$, viewed as a formal object around $Z$ in a similar manner as with the $b$-tangent spaces around $Z$. Moreover, as $\frac{dx_{1}}{x_{1}}(w_{p}) = 1$ for $p \in Z$, we know that $\frac{dx_{1}}{x_{1}} \notin \{ l\in (^{b}T_{p}M)^{*} | l(w_{p}) = 0 \}$, and so we can write

\[ ^{b}T^{*}_{p}M = \begin{cases} 
      T^{*}_{p}M & p \in M\backslash Z \\
      T^{*}_{p}Z\oplus span\{ \frac{dx_{1}}{x_{1}}|_{p} \} & p \in Z \\
   \end{cases}.
\]

\begin{dfn}
    Denote by $^{b}\Omega^{k}(M)$ the space of \textbf{$b$-de Rham $k$-forms}, i.e. sections of $\bigwedge^{k}\left( ^{b}T^{*}M\right)$.
\end{dfn}

We can see that $^{b}\Omega^{k}(M)$ sits inside the usual space of $k$-forms in the following way.  Let $\mu\in\Omega(M)$ be a $k$-form.  We interpret it as a section of $\bigwedge^{k}\left( ^{b}T^{*}M\right)$ by
\begin{itemize}
    \item at $p\in M\backslash Z$ 
    \begin{equation*}
        \mu_{p}\in \bigwedge^{k}\left( T^{*}_{p}M\right) = \bigwedge^{k}\left( ^{b}T^{*}_{p}M\right),
    \end{equation*}
    \item at $p\in Z$
    \begin{equation*}
        \mu_{p} = (i^{*}\mu)_{p}\in \bigwedge^{k}\left( T^{*}_{p}Z\right) \subset \bigwedge^{k}\left( ^{b}T^{*}_{p}M\right),
    \end{equation*}
\end{itemize}
where $i:Z\to M$ is the inclusion map.  For a fixed defining function $f$, i.e. a $b$-map $f:(M,Z)\to(\mathbb{R},0)$, we can write a $b$-de Rham $k$-form $\omega\in ^{b}\Omega^{k}(M)$ as
\begin{equation}\label{EqnbFormDecomp}
    \omega = \alpha\wedge\frac{df}{f}+\beta
\end{equation}
for some $\alpha\in \Omega^{k-1}(M)$ and $\beta\in\Omega^{k}(M)$.  While $\alpha$ and $\beta$ themselves are not unique, their values at every $p\in Z$ are unique.  This decomposition of $b$-de Rham $k$-forms \eqref{EqnbFormDecomp} allows us to extend the exterior derivative to $^{b}\Omega(M)$ by defining its action on $\omega\in$ $^{b}\Omega^{k}(M)$ by
\begin{equation}\label{EqnbExteriorDerivative}
    d\omega = d\alpha\wedge\frac{df}{f} + d\beta.
\end{equation}
This operation is well-defined, and extends smoothly over $M$ as a section of $\bigwedge^{k+1}\left( ^{b}T^{*}\right)$.  Moreover, because the usual exterior derivative satisfies $d^{2}=0$, it is clear by \eqref{EqnbExteriorDerivative} that the extended exterior derivative also satisfies $d^{2}=0$, so we can form the complex of $b$-forms, the \emph{$b$-de Rham complex}
\begin{equation}
    0\rightarrow\ ^{b}\Omega^{0}(M)\xrightarrow{d}\ ^{b}\Omega^{1}(M) \xrightarrow{d}\ ^{b}\Omega^{2}(M)\xrightarrow{d} \cdots \rightarrow 0.
\end{equation}

We can define the notion of symplectic in the $b$-category, and introduce a particular example that will be useful to us.

\begin{dfn}
    Let $(M,Z)$ be a $2n$-dimensional $b$-manifold.  A \textbf{$b$-symplectic form} on $M$ is a $b$-form $\omega \in$ $^{b}\Omega^{2}(M)$ that is closed and non-degenerate, i.e. $d\omega = 0$ and for all $p\in M$, $\omega|_{p}$ is of maximal rank as an element of $\Lambda(^{b}T^{*}_{p}M)$.
\end{dfn}

\begin{remark}
\emph{Given a symplectic manifold $M$, the cotangent bundle $TM$ carries a natural symplectic structure.  In the $b$-setting we can also define a natural $b$-symplectic structure on the $b$-cotangent bundle of a $b$-manifold.\\}

\emph{Let $(M,Z)$ be a $b$-manifold.  Let $\{ x_{1},\dots,x_{n} \}$ be local coordinates on $M$ such that $Z$ is defined by $\{ x_{1} = 0 \}$, $\{ y_{1},\dots,y_{n}\}$ fiber coordinates on $^{b}T^{*}M$.  The canonical one-form is given by}
\begin{equation}
    \theta = y_{1}\frac{dx_{1}}{x_{1}} + \sum_{i = 2}^{n} y_{i}dx_{i},
\end{equation}
\emph{and the corresponding $b$-symplectic form on $^{b}T^{*}_{p}M$ is given by}
\begin{equation*}
    \omega = d\theta = dy_{1}\wedge \frac{dx_{1}}{x_{1}} + \sum_{i = 2}^{n} dy
    _{i}\wedge dx_{i}.
\end{equation*}
\end{remark}


\subsection{Variations of spectral curves}\label{SectionVariations}

Let  $(\mathcal{E}, \phi)$ be an $\mathcal{L}$-twisted Higgs bundle on a Riemann surface $X$ such that $\deg(\mathcal{L}) > \deg(K)$, with spectral curve $S$.  From our work studying the deformation theory of $\mathcal{M}_{X}^{\mathcal{L}}$ in Section \ref{SectionDeformationTheory}, associated to $(\mathcal{E}, \phi)$, we can choose a section $s\in H^{0}(X, K\otimes\mathcal{L})\backslash \{0\}$ which maps $\mathcal{B}_{eff}$ to $\widetilde{\mathcal{B}}\subset \mathcal{B}$, and from it produce the twisted canonical one-form $\Theta$ on $K_{S}$.\\

The chosen section $s\in H^{0}(X,K^{*}\otimes \mathcal{L})\backslash \{0\}$ defines a divisor in $Z\subset X$, namely the divisor of zeroes of $s$.  We want to restrict for now to the case where $s$ has distinct zeroes.  This gives us a relationship $\mathcal{L}\cong K(Z)$.  We can view $X$ together with the divisor $Z$ as a $b$-manifold $(X,Z)$.  Because $S$ lives naturally inside of $K(Z)$, it inherits a natural divisor given by $S\cap \pi^{*}(Z)$ and thus also carries the structure of a $b$-manifold which we will denote as $(S,\pi^*(Z))$ (see Figure \ref{FigTwistedSpectral}).  The projection map $\pi$ restricted to $S$ can be viewed as a map of $b$-manifolds with derivative $d\pi: TS(\pi^{*}(Z)) \to \pi^{*}(TX(Z))\ \in Hom(\pi^{*}(K(Z)),K_{S}(\pi^*(Z)))$, where $K_{S}$ denotes the canonical bundle of $S$.  In this formulation, we have that the twisted canonical one-form is given by $\Theta\coloneqq d\pi(\theta)$, where $\theta$ is the canonical one-form for the log-symplectic structure on $K(Z)$, and $\Theta$ is, in particular, the canonical one-form for the log-symplectic structure on $K_{S}(\pi^{*}(Z))$.\\

\begin{figure}[htp] 
    \centering
    \includegraphics[width=10cm]{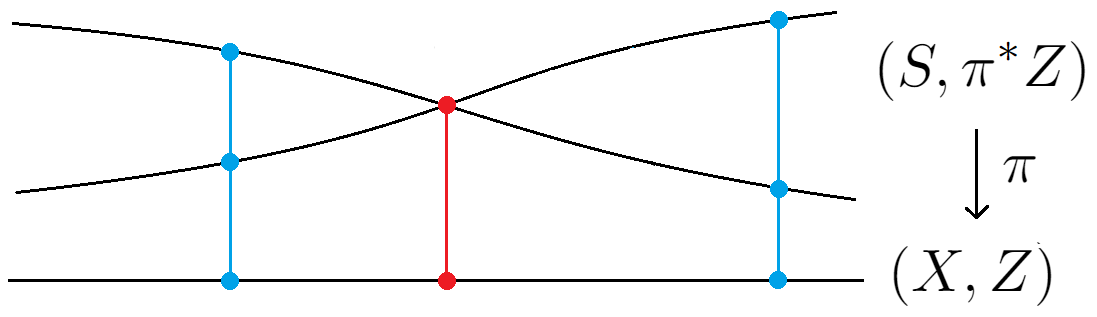}
    \caption{Spectral curve $S$ inside of $K(Z)$.}
    \label{FigTwistedSpectral}
\end{figure}







Let $Z^{reg}_{eff} = \{ (x,b)\in T^{*}X(Z)\times \tilde{\mathcal{B}}^{reg} : p_{b}(x) = 0\}$ be the universal moduli space of spectral curves, where $p_{b}(x)$ is the degree $r$ polynomial with coefficients $b_{1},\dots, b_{r}$.  The space $Z^{reg}_{eff}$ is equipped with two maps: $q:Z^{reg}_{eff} \to \tilde{\mathcal{B}}^{reg}$ the projection onto the second component and $j: Z^{reg}_{eff} \to T^{*}X(Z)$ which maps a spectral curve $\{ p_{b}(x) = 0 \}$ to its image in $T^{*}X(Z)$.
    \begin{center}
\begin{tikzpicture}
  \matrix (m) [matrix of math nodes,row sep=1.5em,column sep=1em,minimum width=2em]
  {
    & Z^{reg}_{eff} & \\    
     S\subset T^{*}X(Z) &  & \tilde{\mathcal{B}}^{reg}\supset U\ni b \\
     X &  & \\};
  \path[-stealth]
    (m-1-2) edge node [above] {$j$} (m-2-1)
            edge node [above] {$q$} (m-2-3)
    (m-2-1) edge node [left] {$\pi$} (m-3-1)

    ;
\end{tikzpicture}
\end{center}

For all $b\in \tilde{\mathcal{B}}^{reg}$ we can choose an open neighbourhood $U$ of $b$ such that in local coordinates $Z_{eff}^{reg}|_{U} \cong U \times S$, where $S$ is the spectral curve viewed as a topological surface with a family of $b$-complex structures $I(t)$ that vary in $U\subset \tilde{\mathcal{B}}^{reg}$.  The composition map $\tilde{\pi} = \pi\circ j: Z^{reg}_{eff}\to X$ corresponds to a family of maps $\pi_{t}:S\to X$ that is holomorphic with respect to the complex structure $I(t)$.  Let $\partial \in T_{b}U$ be a tangent vector.  If we differentiate the condition that $\pi_{t}$ is holomorphic with respect to $I(t)$, we get
\begin{equation}\label{EqnKappa}
    \pi_{*}(\kappa(\partial)) = \overline{\partial}Y,
\end{equation}
\medskip
where $\kappa(\partial) = -\frac{i}{2}\partial I$ is the Kodaira-Spencer class of the deformation $I$, and $Y =\tilde{\pi}_{*}(\partial) \in H^{0}(S,\pi^{*}TX(Z))$, where we have lifted $\partial$ to a vector field on $S$ by assigning to each point on $S$ the vector $\partial$.\\

\medskip


Let $D = \sum_{a}(v(\pi,a)-1)[a]$ be the divisor of ramification points on $S$.  Define $W = \frac{Y}{\tilde{\pi}_{*}} \in H^{0}(S, TS(Z) )(D)$ a $b$-vector field on $S$ with poles along $D$, where we are viewing $\frac{1}{\pi_{*}}$ as a section of $TS(Z)\otimes \pi^{*}(T^{*}X)(Z)$ that is dual to $\pi^{*}$.  From \eqref{EqnKappa}, we have 
\begin{equation}
    \kappa(\partial) = \bar{\partial}W.
\end{equation}
The $b$-vector field $W$ depends on both the choice of $\partial$ and the choice of local differentiable trivialization $U$.  When we wish to show the dependence on $\partial$ we will write $W(\partial)$.\\

Fix a $(1,0)$-vector field $\partial$ on $\mathcal{B}^{reg}_{eff}$. Define the vector field $\delta \in H^{0}(Z^{reg}_{eff},TZ^{reg}_{eff})(D)$ as the unique lift of $\partial$ to $Z^{reg}_{eff}\backslash D$ (i.e. $q_{*}(\delta) = \partial$) such that $\tilde{\pi}_{*}(\delta) = 0$ (i.e. $\tilde{\pi}$ is constant along integral curves of $\delta$).  In a local differentiable trivialization $Z_{eff}^{reg}|_{U} = U \times S$, we can write
\begin{equation}
    \delta = \partial - W(\partial).
\end{equation}


We are interested in differentiating objects on $Z_{eff}^{reg}$ by $\delta$.  Let $V = \chi(\partial)$ be the normal vector field corresponding to $\partial$ 
.


\begin{prop}\label{PropDeltaIndependence}
The variation of $\theta$ with respect to $\delta$ is independent of the trivialization of $Z_{eff}^{reg}$, and is given by
\begin{equation}
    \delta\theta = i_{\hat{V}}d\theta|_{S}.
\end{equation}
\end{prop}


\begin{proof}
Choose a local differentiable trivialization $Z_{eff}^{reg}|_{U} = U \times S$. Define two lifts of $V$ to $T(T^{*}X(Z))$ by $\widetilde{V} = j_{*}(\partial) \in H^{0}(S,T(T^{*}X(Z))$ and $\widehat{V} = j_{*}(\delta) \in H^{0}(S,T(T^{*}X(Z)))(D)$.  We will start by applying $\pi_{*}$ to both $b$-vector fields.  The vector field $\delta$ satisfies $\tilde{\pi}_{*}(\delta) = 0$, and so we have that $\pi_{*}\widehat{V} = \pi_{*}(j_{*}(\delta)) = \tilde{\pi}_{*}\delta = 0$.  
The map $j_{*}: T(U\times S)(Z) \to T(S)(Z)$ restricted to $TS(Z)$ acts as the identity map, so $\widehat{V} = j_{*}(\delta) = j_{*}(\partial - W) = j_{*}(\partial) - j_{*}(W) = \widetilde{V} - W$.  
Applying $\bar{\partial}$ to $\pi_{*}(\widetilde{V}) = \pi_{*}(W)$ yields $\pi_{*}(\bar{\partial}\widetilde{V}) = \pi_{*}(\bar{\partial}W) = \pi_{*}\kappa$.  
The map $\pi_{*}: TS(Z) \to \pi^{*}TX(Z)$ is generically an isomorphism, meaning that $\kappa = \bar{\partial}W = \bar{\partial}V$.  From this, we can say that $\widehat{V}$ is a meromorphic $b-$vector field, i.e. $\bar{\partial}\widehat{V} = 0$ with extra poles at the ramification points of $\pi$.  In the local differentiable trivialization $Z_{eff}^{reg}$, flowing along the vector field $\partial$ produces a family of spectral curves.  Pushing forward $\partial$ to $T^{*}X(Z)$, this family of spectral curves is obtained by flowing along $\widetilde{V} = j_{*}(\partial)$.  Varying $\theta$, which is an object $S$, by the $\partial$ is then given by $\partial\theta = \L_{\widetilde{V}}\theta|_{S}$.  We can understand the variation of $\delta$ on $\theta$ by
\begin{equation}
    \delta\theta = \partial\theta - \L_{W}\theta
    = \L_{\widetilde{V}-W}\theta|_{S}
    = \L_{\widehat{V}}\theta|_{S}
    = \iota_{\widehat{V}}d\theta|_{S} + d(\iota_{\widehat{V}}\theta)|_{S}.
\end{equation}

We can claim that $\iota_{\widehat{V}}\theta$ is a holomorphic one-form.  There are two potential areas of concern: the $b$-geometry, and the poles of $\widehat{V}$.  The interior product of a $b$-vector field with $\theta$ eliminates terms of the form $\frac{dx}{x}$, i.e. eliminating poles that would occur along $Z$.  The canonical one-form $\theta$ vanishes at the ramification points with order at least that of $\pi_{*}$ which eliminates poles coming from $\widehat{V}$.  This means that $\iota_{\widehat{V}}\theta$ is indeed a holomorphic one-form, and so in particular it is constant on $S$, and thus $d(\iota_{\widehat{V}}\theta)|_{S} = 0$.\\




We interpret the term $\iota_{\widehat{V}}d\theta|_{S}$ as being a one-form on $S$, i.e. taking in only tangent vectors to $S$.  In this way, any $S$-tangential component of $\widehat{V}$ will vanish when a tangent vector in $S$ is inserted into the resultant one-form, and so only the $S$-normal component of $\widehat{V}$ will contribute to a nonzero term.  Because $\widehat{V}$ is a lift of $V$ to $T(T^{*}X(Z))$, this normal component is $V$, and so we have $\iota_{\widehat{V}}d\theta|_{S} = \iota_{V}d\theta|_{S}$.



\end{proof}

Let $\xi\in H^{0}(T^{*}X(Z),T(T^{*}X(Z)))$ be the the vector field generating the $\mathbb{C}^{*}$-action on $T^{*}X(Z)$. In local coordinates $(x,y)$, were $x$ is the local coordinate on $X$ and $y$ is the $y$ is the fiber coordinate (i.e. thinking of $(x,y)$ as $(x, y\frac{dx}{x})$), we have that $\xi = y\frac{\partial}{\partial y}$.

\begin{lem}
    If $V, \widehat{V}$ are as in the above proposition, then 
    \begin{equation}
        \widehat{V} = \frac{\alpha}{\theta}\xi,
    \end{equation}
    where $\alpha = \iota_{V}d\theta|_{S}$.
\end{lem}

\begin{proof}
Recall from the proof above that $\widehat{V}$ is a lift of $V$ that satisfies $\pi_{*}(\widehat{V}) = 0$.  Because $\ker(\pi_{*})\cap TS(Z)$ is generically zero, we have that $\widehat{V}$ is characterized by these two conditions.  
Let $V^{*} = \frac{\alpha}{\theta}\xi$.  By uniqueness, it is sufficient to show that $V^*$ satisfies two properties: $V^*$ is a lift of $V$ to $T(T^{*}X(Z))(D)$, and $\pi_{*}(V^*) = 0$.\\

We have that $\pi{*}(V^*) = 0$ for free because $\xi$ is a vector field on $T^{*}X$ in the cotangent fiber direction, i.e. $\xi\in \ker(\pi_{*})$.\\

To check that $V^{*} \in H^{0}(S, T(T^{*}X(Z))(D))$ we need to check the bundle in which each component of $V^*$ lives.  For terms in the ``numerator'', we have $\xi\in H^{0}(T^{*}X(Z), T(T^{*}X(Z)))$, $\alpha = \iota_{V}d\theta|_{S} \in H^{0}(S, T^{*}(T^{*}S(Z)))$.  In the ``demoninator'', we are viewing $\frac{1}{\theta}$ as the dual vector field to $\theta$, i.e. satisfying $\theta(\frac{1}{\theta}) = 1$.  In a local frame with coordinates $(x,y)$ on $T^{*}X(Z)$ where we have $\theta = y\frac{dx}{x}$, we must have that $\frac{1}{\theta} = \frac{1}{y}x\frac{\partial}{\partial x}$.  This local insight tells us that while $\frac{1}{\theta}$ lives in the dual bundle to $\theta$, it has poles at zeros of $\theta$, which occur along the ramification divisor $D$.  
In this way we have $\frac{1}{\theta} \in H^{0}(T^{*}X(Z), T(T^{*}X(Z)))(D)$.  Putting it all together (and restricting to $S\in T^{*}X(Z)$ where necessary), we see that $V^{*} \in H^{0}(S, T(T^{*}X(Z)))(D)$.\\

We need now to check that $V^*$ is a lift of $V$.  We can do this by showing that
\begin{equation*}
    \iota_{V^*}d\theta|_{S} = \iota_{V}d\theta|_{S}.
\end{equation*}

First observe that in a local frame of $T^{*}X(Z)$ with coordinates $(x,y)$, we have that
\begin{align*}
    \theta(\xi) &= y\frac{dx}{x}(y\frac{\partial}{\partial y}) = 0,
\end{align*}
and
\begin{align*}
    \iota_{\xi}d\theta &= (dy \wedge \frac{dx}{x})(y\frac{\partial}{\partial y})\\
    &= y\frac{dx}{x}\\
    &= \theta.
\end{align*}

Computing $\iota_{V^*}d\theta|_{S}$ we have
\begin{align*}
    \iota_{V^*}d\theta|_{S} &= \frac{\alpha}{\theta|_{S}}\iota_{\xi}d\theta|_{S}\\
    &= \frac{\alpha}{\theta|_{S}}\theta|_{S}\\
    &= \alpha = \iota_{V}d\theta|_{S},
\end{align*}

\noindent which shows that $V^{*}$ is a lift of $V$.\\

We have thus shown that $V^{*}$ satisfies the same properties as $\widehat{V}$, and so by uniqueness we have $\widehat{V} = V^{*} = \frac{\alpha}{\theta}\xi$.\\
\end{proof}


We want to use the vector field $\delta$ to differentiate the Bergman kernel, and the twisted analogues of the Eynard-Orantin differentials.  This will produce twisted analogues of certain variational formulas \cite{EynardOrantin07b, BaragliaHuang17}.  We first need to adapt these objects to the $b$-geometric framework.\\

Choose a symplectic basis $\langle A_{1},...,A_{g},B_{1},...,B_{g} \rangle$ for $H_{1}(S,\mathbb{Z})$ such that none of the cycles intersect with $Z$.  The spectral curve comes equipped with a Bergman kernel $B(z_{1}, z_{2})$ as in Definition \eqref{DfnBergmanKernel}. When we make the association $L \cong K(Z)$, and $S$ is realized as the $b$-manifold $(S, \pi^{*}Z)$, we can impose the $b$-structure onto $B$.  In doing so, $B$ becomes a bilinear $b$-differential and so in particular, the local expression of $B$ in neighbourhoods $S\times S$ containing the information of $Z$ will contain $\frac{dz}{z}$-like terms.  As an example, in a neighbourhood of $Z\times Z \subset \Delta_{S\times S}$, we have
\begin{equation*}
    B(z_{1}, z_{2}) = \frac{\frac{dz_{1}}{z_{1}}\frac{dz_{2}}{z_{2}}}{(z_{1} - z_{2})^{2}} + O(1)\frac{dz_{1}}{z_{1}}\frac{dz_{2}}{z_{2}}.
\end{equation*}
Notably, in a neighourhood $U\times V$, where $U$ contains a point in $Z$ and $V$ does not, the local expression will contain a $b$-differential of the form $\frac{dz_{1}}{z_{1}}dz_{2}$, which suggests that $B$ is not symmetric as a $b$-differential.  While this causes no inherent problem for defining $B$, it will become a problem when we try to define the Eynard-Orantin differentials in this setting.  Specifically, the symmetry of the $W_{g,n}$'s come from the symmetry of $B$.  We will remedy this by introducing a symmetrized version of $B$ into this setting.

\begin{dfn}
    Define $\widehat{B}(z_{1}, z_{2}) \coloneqq \frac{1}{2}\left(B(z_{1}, z_{2}) + B(z_{2}, z_{1})\right)$.
\end{dfn}

\noindent Notably, away from $Z$, we can choose a coordinate system $U$ such that $\widehat{B}|_{U} = B|_{U}$.\\

Let $v_{1},\dots,v_{g}$ be a basis of holomorphic differentials, normalized by
\begin{equation}
    \int_{A_{j}}v_{i} = \delta_{ij}.
\end{equation}
We can impose the $b$-structure onto the $v_{i}$, which notably does not affect their normalization over the $A_{i}$-cycles.  From the properties of the ordinary Bergman kernel on $S$, we have that the symmertrized Bergman kernel on $(S,\pi^{*}Z)$ is related to the differentials by
\begin{equation}
    \int_{z_{1}\in B_{j}}\widehat{B}(z_{1},z_{2}) = 2\pi iv_{j}(z_{2}),
\end{equation}
and the period matrix $\tau$ by
\begin{equation}\label{EqnTwistedBergmanandPeriodMatrix}
    \int_{z_{1}\in B_{j}}\int_{z_{2}\in B_{k}}\widehat{B}(z_{1},z_{2}) = \tau_{ij}.
\end{equation}

We now want to differentiate $\widehat{B}$ with respect to the vector field $\delta$, we obtain a twisted version of the Rauch Variational Formula.

\begin{thm}[Twisted-Rauch Variational Formula]\label{ThmTwistedRauch}
If $\partial \in \tilde{B}^{reg}$, $p,q\in S\backslash D$ are distinct points, then
\begin{equation}
    \delta \widehat{B}(p,q) = - \sum_{a\in D}Res_{u\to a}\frac{\delta\theta \widehat{B}(u,p) \widehat{B}(u,q)}{dx(u)dy(u)}
\end{equation}
\end{thm}

\begin{proof}
Choose a local differentiable trivialization $Z_{eff}^{reg}|_{U} = U \times S$ so that $\delta = \partial - W$.  By Theorem \ref{PropDeltaIndependence}, we know that $\delta$ is independent of the choice of trivialization, we can choose the trivialization to be beneficial to us as we see fit.  In particular we will choose the trivialization so that $W$ vanishes in a neighbourhood of $p$ and $q$.  Applying $\delta$ to $\widehat{B}$ in the trivialization we have

\begin{equation*}
    \delta \widehat{B}(p,q) = \partial \widehat{B}(p,q) - \L_{W(p)}\widehat{B}(p,q) - \L_{W(q)}\widehat{B}(p,q) = \partial \widehat{B}(p,q)
\end{equation*}

In this trivialization, and choosing $\kappa = \overline{\partial}W = 0$, the Bergman kernel satisfies the variational formula \cite[pg. 57]{Fay92}
\begin{equation}
    \partial \widehat{B}(p,q) = \frac{1}{2\pi i}\sum_{a}\int_{u\in \gamma_{a}}W(u)\widehat{B}(u,p)\widehat{B}(u,q),
\end{equation}
where the sum is taken over all poles of $W(u)\widehat{B}(u,p)\widehat{B}(u,q)$, specifically, $a$ is a ramification point of $\pi$ (poles of $W$), $a = p$ (pole of $\widehat{B}(u,p)$), or $a = q$ (pole of $\widehat{B}(u,q)$), and $\gamma_{a}$ is a contour around $a$.  Because $W$ vanishes in a neighbourhood of $p$ and $q$, there will be no residue contribution, and we can therefore consider the sum to be over ramification points of $\pi$.  Choosing the contours $\gamma_{a}$ to be sufficiently small, we can ensure that the interiors do not contain any zeros of $s$, and thus need not be concerned with residues coming from $Z$.\\  


In a coordinate chart around a ramification point of $\pi$, we can write $W = -(\delta\theta)\zeta + W'$, where $\zeta = \frac{1}{dydx}$ and $W'$ is smooth.  Using this decomposition for $W$, we have
\begin{align*}
    \delta \widehat{B}(p,q) &= \partial \widehat{B}(p,q)\\
    &= \frac{1}{2\pi i}\sum_{a}\int_{u\in\gamma_a}(-(\delta\Theta)\zeta + W')\widehat{B}(u,p)\widehat{B}(u,q)\\
    &= -\sum_{a}Res_{u\to a}(\delta\Theta)(u)\zeta(u)\widehat{B}(u,p)\widehat{B}(u,q)\\
    &= -\sum_{a}Res_{u\to a}\frac{(\delta\Theta)(u)\widehat{B}(u,p)\widehat{B}(u,q)}{dx(u)dy(u)}
\end{align*}

\end{proof}

\subsection{Twisted topological recursion}\label{SectionTwistedTR}

We next want to define a twisted version of the Eynard-Orantin differentials.  Continuing with the na\"ive approach, we will define the necessary objects as living in $K(Z)$ instead of $K$.

\begin{dfn}
Let $p\in R$.  The \textbf{$\mathcal{L}$-twisted recursion kernel} (associated to $s$) at $p$ is a meromorphic section of $K_{S}(Z)\boxtimes K_{S}(Z)^{*}$ defined by
    \begin{equation}
        K_{p}(z_0,z) = \frac{\int_{t=\alpha}^{z}\widehat{B}(t,z_0)}{(y(z)-y(\sigma_{p}(z))\frac{dx(z)}{x(z)}}
    \end{equation}
where $\alpha$ is an arbitrary base point.
\end{dfn}

\begin{dfn}
The \textbf{$\mathcal{L}$-twisted Eynard-Orantin differentials} (associated to $s$) $W_{g,n}$ are meromorphic sections of the $n$-th exterior tensor product $K_{S}(Z)^{\boxtimes n}$, i.e. multi-$b$-differentials, defined as follows:\\

The initial conditions of the recursion are given by:
\begin{align}
    &W_{0,1}(z) = y(z)\frac{dx(z)}{x(z)}\\ 
    &W_{0,2}(z_1,z_2) = \widehat{B}(z_{1},z_{2}). \label{EqnW02-L}
\end{align}

For all $g,n\in \N$ and $2g - 2 + n \geq 0$, define $W_{g,n}$ recursively by

\begin{align}
    W_{g,n+1}(z_0, \textbf{z}) = \sum_{p\in R}&\mathrm{Res}_{z = p}K_{p}(z_0,z)\Big[ W_{g-1,n+2}(z, \sigma_{p}(z),\textbf{z})\\
    \nonumber &+\sum_{\substack{g_1+g_2=g \\ I\cup J = \textbf{z}}}^{'}W_{g_1, |I|+1}(z,I)W_{g_2, |J|+1}(\sigma_{p}(z), J)\Big] 
\end{align}

where the prime signifies summation excluding the cases $(g_1,I)$ or $(g_2,J) = (0,0)$.\\
\end{dfn}

\begin{remark}\label{RmkTwistedEODifferentials}
    \emph{In the ordinary setting, the Eynard-Orantin differentials satisfy a suite of useful properties, most notably that they are symmetric differentials.  It turns out that the twisted Eynard-Orantin differentials satisfy the same suite of properties.  This is largely because the structure of the recursion revolves around local data, \emph{i.e.} the residues.  The choices made to keep zeroes of $s\in H^{0}(X,K^{*}\otimes L)$ away from ramification points and the symplectic basis of cycles means that the local computations in the recursion do not see the $b$-structure, although the differentials themselves are $b$-objects.  This observation means that there is, in principle, a family of collections of twisted Eynard-Orantin differentials, parametrized by the choice of $s$.  That said, they will all live in different spaces, as the divisor $Z$ depends on $s$.  For each collection of twisted Eynard-Orantin differentials in this family, the proofs of these properties in \cite[Appendix A]{EynardOrantin07b} hold because they depend only on the recursion structure and local computations around the residues.}
\end{remark}  

Let $(\lambda_{1},...,\lambda_{g_{S}})$ be the local singular coordinates on $B_{eff}$, $\partial_{i} \coloneqq \frac{\partial}{\partial \lambda_{i}}$, and $\delta_{i} \coloneqq \delta(\partial_{i})$.

\begin{thm}[Variational Formula for twisted-E-O invariants]
For $g + k > 1$,
\begin{equation}
    \delta_{i}W_{g,k}(p_{1},...,p_{k}) = -\frac{1}{2\pi i} \int_{p\in b_{i}} W_{g,k+1}(p,p_{1},...,p_{k}),
\end{equation}
where the cycle $b_{i}$ is chosen so that it contains no ramification points.
\end{thm}

\begin{proof}
This theorem is essentially the same as Theorem 5.1 of \cite{EynardOrantin07b}, but in the twisted setting.  The original proof only relies on the Rauch variational formula and the diagrammatic representation of $W_{g,n}$.  We proved an analogous form the Rauch variational formula in Theorem \ref{ThmTwistedRauch}.  The diagrammatic representation relies on only the properties of the differentials and the recursion formula, both of which are unchanged in the twisted setting.


\end{proof}

\medskip

To better understand how this variational formula relates to the topology of the twisted setting, let us apply the variational formula to $W_{0,2}(p_{1}. p_{2}) = \widehat{B}(p_{1}, p_{2})$,

\begin{equation}\label{EqnVariationW02}
    \delta_{i}\widehat{B}(p_{1}, p_{2}) = \delta_{i}W_{0,2}(p_{1},p_{2}) = -\frac{1}{2\pi i}\int_{p\in b_{i}} W_{0,3}(p,p_{1},p_{2}).
\end{equation}


Integrating the left hand side twice then yields

\begin{align*}
    \int_{p_{1}\in b_{j}}\int_{p_{2}\in b_{k}} \delta_{i}\widehat{B}(p_{1}, p_{2}) &= \partial_{i} \int_{p_{1}\in b_{j}}\int_{p_{2}\in b_{k}} \widehat{B}(p_{1}, p_{2})\\
    &= 2\pi i\partial_{i}\tau_{jk}\\
    &= 2\pi i c_{ijk},
\end{align*}

written more cleanly

\begin{equation*}
    c_{ijk} = \frac{1}{2\pi i}\int_{p_{1}\in b_{j}}\int_{p_{2}\in b_{k}} \delta_{i}\widehat{B}(p_{1}, p_{2}).
\end{equation*}


Utilizing \eqref{EqnVariationW02}, we have a relationship between the twisted Donagi-Markman cubic and $W_{0,3}$ given by
\begin{equation}\label{EqnTwistedDMCW03}
    c_{ijk} = -\left(\frac{1}{2\pi i}\right)^{2}\int_{p_{1}\in b_{j}}\int_{p_{2}\in b_{k}}\int_{p\in b_{i}} W_{0,3}(p,p_{1},p_{2}).
\end{equation}

We can continue the computation on the right-hand side by adapting a lemma from \cite[Appendix A]{EynardOrantin07b}.  The proof of this lemma again only depends on the properties of twisted Eynard-Orantin differentials.

\begin{lem}
\begin{equation}
    W_{0,3}(p, p_{1}. p_{2}) = \sum_{a}Res_{q\to a}\frac{B(p,q)B(p_{1},q)B(p_{2},q)}{dx(q)dy(q)}
\end{equation}
where the sum is taken over ramification points $a$ of the spectral curve.
\end{lem}

From this lemma we obtain
\begin{align*}
    c_{ijk} &= -\left(\frac{1}{2\pi i}\right)^{2}\sum_{a}\int_{p_{1}\in b_{j}}\int_{p_{2}\in b_{k}}\int_{p\in b_{i}} Res_{q\to a}\frac{B(p,q)B(p_{1},q)B(p_{2},q)}{dx(q)dy(q)}\\
    &= -2\pi i \sum_{a}Res_{q\to a}\frac{v_{i}(q)v_{j}(q)v_{k}(q)}{dx(q)dy(q)}.
\end{align*}

This proves a local analogue of the residue formula for the Donagi-Markman cubic as presented in Baralgia-Huang.

\begin{lem}[Local residue formula for the twisted Donagi-Markman cubic]
\begin{equation*}
    c_{ijk} = 2\pi i \sum_{a}Res_{q\to a}\frac{v_{i}(q)v_{j}(q)v_{k}(q)}{dx(q)dy(q)}
\end{equation*}
\end{lem}

We want to return our attention back to \eqref{EqnTwistedDMCW03}.  We can continue this process of applying the variational formula to $W_{0,2}$ multiple times with different choices of $\delta_{i}$ and obtain:

\begin{thm}\label{ThmBigThm}
    \begin{equation}        \partial_{i_{1}}\partial_{i_{2}}\dots\partial_{i_{m-2}}\tau_{i_{m-1}i_{m}} = -\left( \frac{i}{2\pi} \right)^{m-1} \int_{p_{i_{1}}\in b_{i_{1}}}\dots\int_{p_{m}\in b_{i_{m}}} W_{0,m}(p_{1},...,p_{m}).
    \end{equation}
\end{thm}

We can interpret this theorem in the following way.  The Higgs bundle $(\mathcal{E}, \phi)$ produces a spectral curve $S$.  This spectral curve lives over a point $b\in \widetilde{\mathcal{B}}^{reg}$, the image of the deformation space of spectral curves inside of the Hitchin base induced by the chosen section $s\in H^{0}(X,K^{*}\otimes \mathcal{L})$. The $g=0$ twisted Eynard-Orantin variants compute the Taylor series expansion of the period matrix at the point $b$.  This means that with the information of a single spectral curve $S$, we can use the twisted Eynard-Orantin invariants to understand local deformations of $S$ through their period matrices.\\

Throughout our development, we have started with a twisted Higgs bundle and then considered its spectral curve, as a unique global object.  Our results here also allow us to draw conclusions in the setting where one starts with a spectral curve presented in a local form independent of any reference to a given Higgs bundle, as is typically the case in topological recursion. If we start with the data of a local spectral curve, considered as a monic polynomial of degree $r$ in $y(x)$ with coefficients as polynomials in the affine coordinate $x$, then we can interpret this form as the spectral curve of a rank-$r$ twisted Higgs bundle on $\mathbb P^1$, although of course not uniquely so: the underlying bundle $\mathcal E$ and the twist $\mathcal L$ are not uniquely determined and, even when these are fixed, there is still typically a family of such isomorphism classes, corresponding to a subvariety of the associated Hitchin fibre.  (We can ensure that the resulting Higgs bundles are $\mathfrak{sl}(r,\mathbb C)$-valued by asking for the coefficient of $y^{r-1}$ to be identically zero.) Of particular importance to our set up is that there are countably-infinitely-many line bundles $\mathcal L$ in which the Higgs field can take values --- there is a lower bound on the degree of $\mathcal L$ determined by the degrees of the spectral data as polynomials in $x$, but there is no upper bound. As such, one can ask if the twisted generalization of the topological recursion is sensitive to the various embeddings of the spectral curve into $\mbox{Tot}(\mathcal L)$ or if the results are independent of the choice of how to interpret the spectral curve as a Hitchin spectral curve.  For this, it is worth noting that while the coordinates $\lambda_{i}$ depend on $s$, this result would hold for any $s$ that does not have zeros along the symplectic basis of cycles.  This aligns with the observation in Remark \ref{RmkTwistedEODifferentials}, as the variational formula only depends on local data away from the zeroes of $s$.  While Theorem \eqref{ThmBigThm} will hold for any choice of $s$ compatible with the symplectic basis of cycles, we have to be careful about the interpretation between various choices of $s$.  The period matrix is an inherent object on $S$, depending only on the choice of symplectic basis, regardless of where or how we view $S$ as residing.  It makes sense, then, that choices of $s$ should not change the period matrix --- this is true, as we write the period matrix in terms of $W_{0,2}=\widehat{B}$ by \eqref{EqnTwistedBergmanandPeriodMatrix}, for any compatible $s$.  The information of the Taylor series, however, gives information about how the period matrix changes with respect to coordinates on $\widetilde{\mathcal{B}}^{reg}$, which does, \emph{a priori}, depend on $s$ (it has also not been addressed in the case of the twisted invariants).  This dependence on $s$ is a subtly that has not been addressed.  In principle, deformations are controlled by the effective Hitchin base $\mathcal{B}_{eff}$, which is independent of $s$, rather than $\widetilde{\mathcal{B}}$.  We would like to say that the full interpretation of Theorem \ref{ThmBigThm} is independent of $s$, however, without further understanding the dependence of $\widetilde{B}$ on $s$, we leave this statement as a conjecture for now.
\\


\section{Further perspectives on $\mathcal{L}$-twisted Hitchin geometry}

The comments at the end of Section \ref{SectionTwistedTR} suggest that we need to better understand the relationship between the various choices of $s\in H^{0}(X,K^{*}\otimes\mathcal{L})$ and the geometry of the spectral curve.  To study this dependence, we consider the vector bundle over $H^{0}(X,K^{*}\otimes\mathcal{L})$ with fibres $\mathcal{B}$ (see Figure \ref{FigSpectralBundle}).  Because $H^{0}(X,K^{*}\otimes\mathcal{L})$ is a vector space, this is just the trivial bundle $\mathcal{B}\times H^{0}(X,K^{*}\otimes\mathcal{L})$.

\begin{figure}[htp] 
    \centering
    \includegraphics[width=10cm]{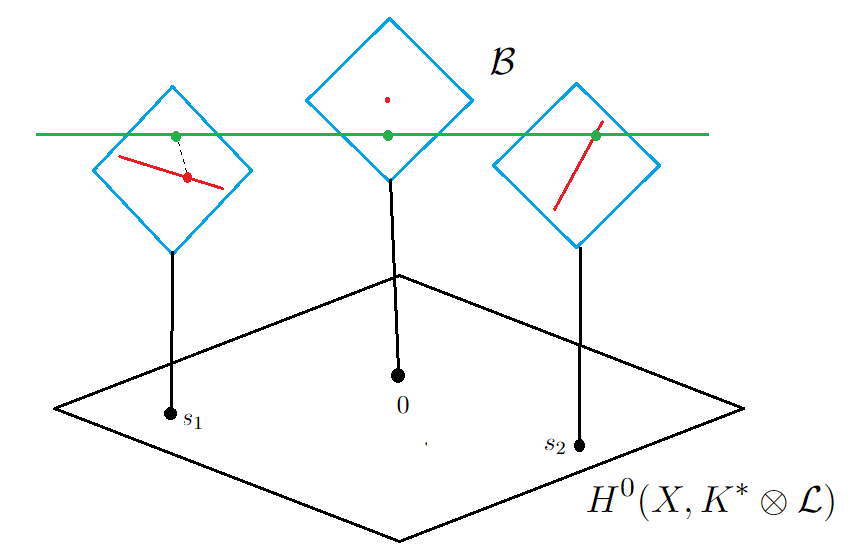}
    \caption{Vector bundle over $H^{0}(X,K^{*}\otimes L)$ whose fibres are $\mathcal{B}$ (in blue).  Each fibre has a distinguished subspace $\widetilde{\mathcal{B}}_{s}$ (in red).  For a chosen Higgs bundle, there is a constant section given by the characteristic coefficients (in green).}
    \label{FigSpectralBundle}
\end{figure}

Let $(\mathcal{E}, \phi)$ be an $\mathcal{L}$-twisted Higgs bundle with spectral curve $S$.  Choose a symplectic basis $\langle A_{1},...,A_{g},B_{1},...,B_{g} \rangle$ for $H_{1}(S,\mathbb{Z})$.  There is a natural constant section on this vector bundle associated to $(\mathcal{E}, \phi)$, which assigns to each point in the base the characteristic coefficients of the Higgs field in the fibre.  In each fibre, there is also an identified subspace given by $\widetilde{\mathcal{B}}_{s}$, the image of $\mathcal{B}_{eff}$ under multiplication by $s$ inside of $\mathcal{B}$.  Recall from Section \ref{SectionDeformationTheory} that $\mathcal{B}_{eff}$, and by extension $\widetilde{\mathcal{B}}_{s}$, is the space of deformations of the spectral curve.  Away from the point $s\equiv 0$ where $\dim\widetilde{\mathcal{B}}_{s} = 0$, the dimension of each $\widetilde{\mathcal{B}}_{s}$ is equal, although they do not define the same subspaces.  Notably, the constant section defined by $(\mathcal{E},\phi)$ need not intersect the identified subspace in the fibre.  There is a subset of points in the base whose zeroes intersect the symplectic cycles.  We will call this set the \emph{incompatible locus}.  Away from the incompatible locus (\emph{i.e.}, along points in the base whose zeroes do not intersect the symplectic cycles) we can, in a holomorphic way, identify $S$ with a point on $\widetilde{\mathcal{B}}_{s}$, which describes $S$ with ``zero deformations''.  Like this, we have a natural section on the base minus the incompatible locus coming from $S$.\\

To determine the dependence on $s$, we need to understand three subsets of $H^{0}(X,K^{*}\otimes\mathcal{L})$: the incompatible locus, the set of sections with non-distinct zeroes, and the set of sections whose zeroes overlap with ramification points of $S$.  Removing these subsets from $H^{0}(X,K^{*}\otimes\mathcal{L})$ will yield a space of sections that are suitable for the constructions in the previous section.  The topology of this space will determine which sections are ``equivalent'', in the sense that we can move from one to another by moving zeroes of $s$ without crossing any of the noted subsets above. We can then look at a slightly weakened version of the conjecture made at the end of the previous chapter: the interpretation of Theorem \ref{ThmBigThm} holds on equivalence classes of sections.\\

There are also two directions that we can proceed to further investigate this vector bundle:
\begin{enumerate}
    \item How do sections that pass through the identified subspaces in the fibres correspond to spectral curves?  For example, can a section defined over only part of $H^{0}(X,K^{*}\otimes\mathcal{L})$ define a incompatible locus, and thus produce a set of cycles on $S$?  Can it make that set of cycles a symplectic basis?  
    \item We can pull this bundle back onto itself.  The tautological section on this pullback bundle contains the information of all spectral curves.  Can we use this to construct a \emph{universal spectral curve}, which contains all the important data of spectral curves?  Suppose we have an $\mathcal{L}$-twisted Higgs bundle $(\mathcal{E}, \phi)$.  Does this universal spectral curve see the properties of the spectral curve that are invariant under choices of $s$? (Or compatible choices of $s$?)
\end{enumerate}

\vspace{5pt}
\begin{remark}
    \emph{It is important to bring attention to an additional choice that was made at the beginning of Section \ref{SectionVariations}, specifically, only choosing $s$ to have distinct zeroes.  This choice was made to ensure that we could work in the context of $b$-geometry, and that our Higgs fields only have simple poles.  If we dropped this condition, we would be entering the setting of $b^{k}$-geometry \cite{Scott16} and wild Higgs bundles \cite{BiquardBoalch01, Boalch18, FredricksonNeitzke21}.  Assuming that we could carry out similar calculations in this setting, it would open up a larger class of sections in $H^{0}(X,K^{*}\otimes\mathcal{L})$, which correspond to situations where we allow zeroes to cross as we vary $s$.  We will not investigate these ideas here, but wish to bring attention to them as a further direction of generalization.}
\end{remark}


\bibliography{refs}{}
\bibliographystyle{acm}

\end{document}